\documentclass[12pt,leqno]{article}

\usepackage{amsmath}
\usepackage{amsthm}
\usepackage{amssymb}
\usepackage{amsfonts}
\usepackage{latexsym}
\usepackage{indentfirst}
\usepackage{graphicx}
\usepackage{color}
\usepackage[english]{babel}
\usepackage{paralist}
\usepackage{fullpage}
\usepackage{color}
\usepackage{stmaryrd}
\usepackage{amsfonts}
\usepackage{hyperref}
\usepackage{cleveref}
\date{}

 \newcommand{\co}{{\mathcal O}}  
 \newcommand{\oco}{{\overline{\mathcal O}}} 
\newcommand{\oph}{{\overline{\Phi}}}

\newcommand{\ca}{{\mathcal A}}    \newcommand{\cb}{{\mathcal B}}   
\newcommand{\cC}{{\mathcal C}}   
\newcommand{\cd}{{\mathcal D}}
\newcommand{\ce}{{\mathcal E}}    \newcommand{\cf}{{\cal F}}   \newcommand{\cg}{{\mathcal G}}

\newcommand{\calr}{{\mathcal R}}     \newcommand{\ct}{{\mathcal T}}
\newcommand{\cu}{{\mathcal U}}       \newcommand{\cv}{{\mathcal V}}

\newcommand{\Q}{\mathbb{Q}}
\newcommand{\Pp}{\mathbb{P}} 
\newcommand{\Tt}{\mathbb{T}} 
\newcommand{\Ee}{\mathbb{E}} 
\newcommand{\Ss}{\mathbb{S}}

\newcommand{\al}{{\alpha}}

\newcommand*{\Fscr}{\mathcal F}

\newcommand*{\cE}{\mathcal{E}}

\renewcommand*{\d}{\, \mathrm{d}}

\newcommand{\ds}{\displaystyle}

\numberwithin{equation}{section}

\providecommand\mathbb{\bf}
\newcommand\R{{\mathbb R}}

\providecommand\text[1]{\textrm{#1}}

\newtheorem{thm}{Theorem}[section]
\newtheorem{lem}[thm]{Lemma}
\newtheorem{prop}[thm]{Proposition}
\newtheorem{cor}[thm]{Corollary}
\newtheorem{rem}[thm]{Remark}  %remarca numerotata

\numberwithin{equation}{section}

\begin{document}

\noindent
{\bf\large Continuous flows driving  Markov processes and multiplicative $L^p$-semigroups}\\[1mm]

\noindent
Lucian Beznea\footnote{Simion Stoilow Institute of Mathematics  of the Romanian Academy, Research unit No. 2,  
P.O. Box \mbox{1-764,} RO-014700 Bucharest, Romania, and  University POLITEHNICA Bucharest, CAMPUS Institute.
{E-mail}:
{lucian.beznea@imar.ro}}, 
Mounir Bezzarga\footnote{Institut Pr\'eparatoire aux Etudes d'Ing\'enieurs de Tunis, 2 Rue Jawaher Lel Nehru, 1008\ Montfleury-Tunis, Tunisia. %\newline
{E-mail}: 
{mounir.bezzarga@ipeim.rnu.tn} },
and
Iulian C\^{i}mpean\footnote{ University of Bucharest, Faculty of Mathematics and Computer Science, and Simion Stoilow Institute of Mathematics of the Romanian Academy, Research unit No. 2, P.O. Box \mbox{1-764,} RO-014700 Bucharest Romania. E-mail: iulian.cimpean@unibuc.ro}

\vspace{10mm}

\noindent{\small{\bf  Abstract.} 
We develop a method of driving a Markov processes through a continuous flow.
In particular, at the level of  the transition functions we investigate an approach of adding a first order operator to the generator of a Markov process, when the two generators commute.
A relevant example is  a measure-valued   superprocess having a continuous flow as spatial motion and a branching mechanism which does not depend on the spatial variable.
We  prove that any flow is actually continuous in a convenient topology 
and we show that a Markovian multiplicative semigroup on an $L^p$ space is generated by a continuous flow, completing the answer to the question whether it is enough to have a measurable structure, like a $C_0$-semigroup of Markovian contractions on an $L^p$-space with no fixed topology, in order to esnsure the existence of a right Markov process associated to the given semigroup.
We extend from bounded to unbounded functions the  weak generator (in the sense of Dynkin) and the corresponding martingale problem.}

\vspace{4mm}


% 60J35, % Transition functions, generators and resolvents
% 37C10 % Dynamics induced by flows and semiflows
% 60J45, % Probabilistic potential theory
% 47D07 % Markov semigroups and applications to diffusion processes
% 60J40  % Right processes
% 60J68, % Superprocesses

% 35J60, % Nonlinear elliptic equations
% 60J80, % Branching processes (Galton-Watson, birth-and-death, etc.)
% 37H05 General theory of random and stochastic dynamical systems
% 37L55 Infinite-dimensional random dynamical systems; stochastic equations
\vspace{1mm}


\section{Introduction} %  Section 1, Introduction

The solution of a first order differential equation in an Euclidean domain $E$, a typical example of continuous flow on $E$, 
may be regarded as a deterministic Markov process and its generator $D$ acts on functions on $E$ as a derivation, i.e., $D(u^2)=2uDu$.
It turns out this property  remains valid for the generator of a right continuous flow on a general state space $E$,  hence the approach herein considered provides 
a substitute for a gradient type operator in a general setting, possible infinite dimensional. 

The purpose of this work is twofold.
First, we study Markov processes which are driven by continuous flows, namely processes $X^\Phi$ admitting the structure
\begin{equation} \label{eq1.1}
X^\Phi_t= \Phi_t (X_t), t\geqslant 0,
\end{equation}
where $\Phi$  is a continuous flow  and $X$ is a Markov process on $E$.
Second, we investigate multiplicative semigroups in an $L^p$-context and the associated continuous flows, completing the answer given in  \cite{BeBoRo06} to the question whether it is enough to
have a measurable structure, like a $C_0$-semigroup of  Markovian  contractions on an $L^p$-space, with no fixed topology,  in order to find a Markov process behind the given semigroup;  see also \cite{BeBoRo06a}, and \cite{BeCiRo18}.
We show that the additional property of being multiplicative on $L^p$ (or equivalently, the $L^p$-generator to be a derivation) is enough for the existence of a continuous flow having the given $L^p$-semigroup as its transition function.

If $L$ (resp. $L^\Phi$) is the generator of $X$ (resp. $X^\Phi$) and  $(\ref{eq1.1})$ holds, then $L^\Phi = L+ D$, so, we regard 
$L^\Phi$  as a modification of $L$ with a drift type operator $D$.  
In this way, the weak generator (in the sense of E.B. Dynkin) of a Markov process 
but also of a right continuous flow are main tools in our approach. 
An example for which our method apply is obtained by taking $L$ to be the fractional power (or more general, a Bochner subordination) of $D$.
We present in particular a method of extending the domain of the weak generator from bounded to unbounded functions, 
enlarging the class  of functions for which the associated martingale problem has a solution; for other related extensions of the weak generator see
\cite{HiYo12} and \cite{Ku20}.

The motivation for the first aim is  the application to the  measure-valued  superprocesses, cf. e.g.  \cite{Li11}.
Recall that the  state space of a superprocess $\widehat{X}$  is the set 
 $M(E)$ of all positive finite measures on $E$ and the evolution is given by a branching mechanism 
and  a spatial motion which describe the movement of the particles between the branching moments.
If the spatial motion is  a right continuous flow and the branching mechanism 
does not depend on the spatial variable then the representation $(\ref{eq1.1})$ holds on $M(E)$ 
by means of a second superprocess $\widehat{X^0}$  and of the flow on measures induced by $\Phi$,
$$
\widehat{X}_t %= \widehat{\Phi}_t (\widehat{X^0_t})
= \Phi_t( \widehat{X_t^0}), t\geqslant 0.
$$
Here, the superprocess $\widehat{X^0}$ is such that it has the same branching mechanism as $\widehat{X}$, however,  it has no a spatial motion.

The structure and main results of the paper are as follows.

In Section \ref{sec2} we present the basic facts on the right continuous flows and  flows on a space with no fixed topology, called
semi-dynamical systems.
Theorem \ref{thm2.2}  shows that actually such a flow is continuous in a convenient topology, extending a result from
\cite{Sh88}. As a consequence, the induced capacity is tight.

The results on the extended weak generator of a Markov process are exposed in Section \ref{sec3}, including the associated martingale problem.
In  Subsection \ref{subsec3.1} we study  the extended weak generator of a semi-dynamical system.
Finally, we show in  
Subsection \ref{subsec3.2}, Proposition \ref{prop3.8},  
that a continuous flow may be stopped at the first entry time in the complement of an open set, a procedure  already  used in \cite{BeIoLu-St21} and \cite{BeLu-StTeod24}.
Several technical proofs are included in the Appendix.

The theory of continuous flows driving Markov process is investigated in Section \ref{sec4}. 
The main result (Theorem \ref{thm4.1}) about the representation $(\ref{eq1.1})$ and the drift modification of the weak generator of Markov process,  
is followed by   the example  on  the Bochner subordination of a right continuous flow,  
stated in Corollary \ref{cor4.4} from Subsection \ref{subesec4.1}.
%(Example \ref{exam4.3} and Corollary \ref{cor4.4}).
The main application in this framework is given in 
 %Example \ref{exam4.5}.
 Subsection \ref{subesec4.2}.

Theorem \ref{thm5.3}  from \Cref{sec5} is the central result that relates multiplicative $L^p$-semigroups with continuous flows.

\section{Semi-dynamical systems and right continuous flows}\label{sec2}  % Section 2

\paragraph{Transition functions, resolvent of kernels, and excessive functions.}
Let $(E,\mathcal B(E))$ be a Lusin measurable space, i.e., it is
measurable isomorphic to a Borel subset of a metrizable compact
space endowed with the Borel $\sigma$-algebra.

For a $\sigma$-algebra $\cg$ we denote by 
$ [\cg]$  (resp. $p\cg$) the vector space of all real-valued (resp. the set of all positive, numerical) $\cg$-measurable functions on $E$.
 Also,  for a set of real-valued functions $\cC$ we denote by $\sigma(\cC)$ the $\sigma$-algebra generated by $\cC$,
 by $[\cC]$ the vector space spanned by $\cC$, and by $p\cC$ (resp. $b\cC$) the set of all positive (resp. bounded) functions from $\cC$.

We consider a sub-Markovian resolvent of kernels  
$\mathcal U=(U_\alpha)_{\alpha>0}$ on $(E,\mathcal B(E))$.
A nonnegative, numerical, $\mathcal{B}(E)$-measurable function defined on $E$ is called {\it $\cu$-excessive} provided that
\begin{equation}\label{eq:ex}
    \quad{\alpha}U_{\alpha}u \leqslant u \quad \mbox{for all } {\alpha} > 0, \quad \mbox{ and } \quad \lim_{\alpha\to\infty}{\alpha}U_{\alpha}u(x)=u(x), 
 x\in E.
\end{equation}
We  denote by $\mathcal E(\mathcal U)$ the set of all real-valued
$\cu$-excessive functions. 
If $\beta>0$ we denote by $\mathcal
U_\beta$ the sub-Markovian resolvent of kernels
$(U_{\beta+\alpha})_{\alpha>0}$. 
%Furthermore,  for a set of functions $\cf$ we denote the subset of all its bounded and non-negative elements by $b\cf$ and $\cf_+$ respectively.
A $\mathcal{U}_\beta$-excessive function is also called {\it $\beta$-excessive.}
If $w$ is a $\mathcal{U}_{\beta}$-supermedian function 
({\it i.e.,} ${\alpha}U_{{\beta}+{\alpha}}w\leqslant w$ for all ${\alpha}>0$), then its {\it $\mathcal{U}_{\beta}$-excessive regularisation} 
$\widehat{w}$ is given by $\widehat{w}(x):=\sup_{\alpha} {\alpha}U_{{\beta}+{\alpha}}w(x)$,  $x{\in E}$.

\vspace{1mm}

Let $\Tt =(T_t)_{t\geqslant 0}$ be a sub-Markovian {\it transition function}  on 
%a Lusin measurable space 
$(E, \cb(E))$, that is
\begin{enumerate}
    \item [-] $T_t$
is a sub-Markovian kernel on $E$, $T_0=Id$, $T_t\circ T_s=T_{t+s}$ for all $t, s >0$;
\item[-] for every $f\in bp\cb(E)$ the mapping
$(x, t)\rightarrow T_t f(x)$ is $\cb(E) \otimes \cb(\R_+)-$measurable.
\end{enumerate} 
Let further  $\cu=(U_\alpha)_{\alpha > 0}$ be the resolvent of sub-Markovian kernels induced by $\Tt =(T_t)_{t\geqslant 0}$,
\begin{equation*}
    U_\alpha := \int_0^\infty e^{-\alpha t} T_t \d t, \mbox{  for all } \alpha >0,
\end{equation*}
and let $U$ be the {\it potential kernel} of  $\Tt $  (and of  $\cu$),
$U := \int_0^\infty  T_t \d t$.

\noindent{Recall} that condition \eqref{eq:ex} is equivalent with
\begin{equation*}
    T_t u\leqslant u  \mbox{ for all } t>0 \mbox{ and } \lim_{t \searrow 0} T_t u(x)= u(x) \mbox{ for all } x\in E.
\end{equation*}
If $\beta>0$ then clearly, $\mathcal{U}_{\beta}$ is the resolvent of kernels induced by the sub-Markovian transition function 
$\Tt_\beta =(e^{-\beta t} T_t)_{t\geqslant 0}$.
Notice that the potential kernel of $\Tt_\beta$ is the bounded kernel $U_\beta$, in contrast with the potential kernel $U$ of
$\Tt$ which might be an unbounded kernel.

Assume now that $E$ is a Lusin topological space ({i.e.,} $E$ is homeomorphic to a Borel subset of a metrizable compact space) 
and let $\mathcal{B}(E)$ its  Borel ${\sigma}$-algebra.  
Let $X=(\Omega,\cf,\cf_t, X_t, \Pp^x,\zeta)$ be a right Markov process on $E$ having $(P_{t})_{t\geqslant 0}$ 
as transition function, hence
\begin{equation*}
    P_{t}f(x)=\mathbb{E}^{x}(f(X_{t}),t<\zeta),\, {t\geqslant 0},\, f\in p\mathcal{B}(E),
\end{equation*}
and let $\cu =(U_\alpha)_{\alpha>0}$ be the resolvent on $(E,\cb(E))$ associated with $(P_{t})_{t\geqslant 0}$.
The {\it fine topology}  is the coarsest topology on $E$  making continuous all $\beta$-excessive
functions for some (and equivalently for all)  $\beta>0$.
Recall that in this context, a function $f$ from $p\cb(E)$  is finely continuous if and only if 
$t\rightarrow f(X_t)$ is a.s. right continuous on $[0,\zeta)$.
Using this characterization and the fact that $X$ is has right continuous paths, any  continuous function on $E$ is also finely continuous.

\vspace{1mm}

\paragraph{Semi-dynamical systems.} Let $(E, \cb)$ be a Lusin measurable space
and let $\Phi=(\Phi_t)_{t\geqslant 0}$ be a family of mappings $\Phi_t:E\rightarrow E, t\geqslant 0$.
Then $\Phi$ is called  {\it semi-dynamical system} on $E$ provided that the  following conditions are satisfied:
\begin{enumerate}
    \item[(sd1)] $\Phi_{t+s}(x)= \Phi_t(\Phi_s(x))$  for all $s,t > 0$ and $x\in E$;
    \item[(sd2)] $\Phi_{0}(x)=x$    for all  $x\in E$;
    \item[(sd3)] For each $t>0$ the function $ E \ni x \longmapsto \Phi_t(x)$ is $\cb(E)/ \cb(E)$-measurable;
    \item[(sd4)] There exists a countable set $\cC_o\subset bp\cb$ such that $\cC_o$ separates the points of $E$ and 
$\lim_{t\searrow 0} f(\Phi_t(x))= f(x)$ for all $x\in E$ and $f\in \cC_o$.
\end{enumerate}

In the sequel, if $f\in [\cb]$ and $N$ is a kernel on $(E, \cb(E))$, then by $Nf\in [\cb(E)]$ we mean that
 $N|f|<\infty$, hence $N(f^+)$ and $N(f^-)$  are real-valued functions and $Nf=N(f^+)- N(f^-).$

\begin{rem}
Note that if $\Phi=(\Phi_t)_{t\geqslant 0}$ is a semi-dynamical system  on $E$ then the function $ E\times [0,\infty) \ni (x,t) \longmapsto \Phi_t(x)$ is $\cb(E) \otimes\cb([0,\infty))/\cb(E) $-measurable. 
This follows by a monotone class argument, observing first that from $(sd4)$ it follows  that for every $f\in \cC_o$ the real-valued function
$t\longmapsto  f(\Phi_t(x)) $ is right continuous on $[0, \infty)$.
\end{rem}

For each $t\geqslant 0$ define the Markovian kernel on $E$ as
$$
S_t f :=  f \circ \Phi_t \, \mbox{ for all } f\in p\cb(E). 
$$
Then the  family $\Ss=(S_t)_{t\geqslant 0}$ is a Markovian transition function on $E$, called the {\it transition function of } the semi-dynamical system 
$\Phi=(\Phi_t)_{t\geqslant 0}$. %Notice that  the transition function  $(S_t)_{t\geqslant 0}$ 
 
\begin{rem} \label{rem2.1} % Remark 2.1
\begin{enumerate}
    \item[(i)] The transition function $\Ss=(S_t)_{t\geqslant 0}$ of  a semi-dynamical system   $\Phi=(\Phi_t)_{t\geqslant 0}$ on $E$ is {\rm multiplicative}, that is, $S_t(f g)= (S_t f)(S_t g)$ for all $t\geqslant 0$ and $f,g \in bp\cb(E)$.
    \item[(ii)] It is known that the converse of assertion (i) holds: Let $\Ss=(S_t)_{t\geqslant 0}$  be a Markovian transition function on $E$ which is multiplicative and
    \begin{equation} \label{eq2.1}
    \mbox{there exists a countable set $\cC_o\subset bp\cb$ such that $\cC_o$ separates the points of $E$,}
    \end{equation}
    and $\lim_{t\searrow 0} S_t(x) = f(x)$ for all $x\in E$ and $f\in \cC_o$.
    Then there exists a semi-dynamical system on $E$, having the transition function $\Ss$.\\
    Indeed, for $x\in E$ and $t\geqslant 0$ let $S_{t,x}$ be the probability on $E$  induced by the measure $f\longmapsto S_t f(x)$. 
    If $A\in \cb(E)$ then, $S_t$ being multiplicative, we have $S_{t,x}  (1_A)= (S_{t,x}  (1_A))^2$, so, either $S_{t,x}  (1_A)=0$ or $S_{t,x}  (1_A)=1$. 
    It follows that there exists $\Phi_t(x)\in E$ such that $S_{t,x}  = \delta_{\Phi_t(x)}$. Since $S_t f\in bp\cb(E)$ for all $f\in  bp\cb(E)$ it follows that $(sd3)$ holds. 
    The semigroup property of $(S_t)_{t\geqslant 0}$  implies that $(sd1)$ is verified and from $(\ref{eq2.1})$ it follows that $(sd4)$ also holds. 
    Finally, because $S_0=Id$ we get $(sd2)$.
    \item[(iii)] Let $\ca$ be a collection of bounded real-valued functions defined on $E$ which is multiplicative (i.e.,  if $ f, g\in \ca$ then $fg \in \ca$)  and generates $\cb(E)$. 
    Let  further $\Ss=(S_t)_{t\geqslant 0}$ be a sub-Markovian transition function on $E$ such that $S_t(f g)= (S_t f)(S_t g)$ for all $f,g \in \ca$. 
    Then $\Ss=(S_t)_{t\geqslant 0}$ is multiplicative. Indeed, if we fix  $x\in E$ and $g \in \ca$ then,  writing $g=g^+ - g^-,$ the functionals $f\longmapsto S_t (fg)(x)$ and $f\longmapsto S_t (f)(x) S_t(g)(x)$ are differences of two positive finite measures which coincide on $\ca$. By a monotone class argument we get $S_t(f g)= (S_t f)(S_t g)$ for all $f\in bp\cb(E)$. 
    Fixing now $f\in bp\cb(E)$ and arguing as before, we conclude that the last equality holds for all $f, g \in bp\cb(E)$. 
    \item[(iv)] We have
    \begin{equation}\label{eq2.2}
     \mbox{ if } \Ss= (S_t)_{t\geqslant 0} \mbox{ is multiplicative and } v \mbox{  is }\beta\mbox{-excessive then } v^2 \mbox{ is }  {2\beta}\mbox{-excessive,} 
    \end{equation}
    where $\beta \geqslant 0$.
    Indeed, since $\Ss =(S_t)_{t\geqslant 0}$ is multiplicative
    we have $e^{-2\beta t} S_t (v^2)= (e^{-\beta t} S_t v)^2\leqslant v^2$, where the inequality holds because $v$ is $\beta$-excessive.
    Then clearly 
    $\lim_{t\searrow 0}  e^{-2\beta t} S_t (v^2)=  \lim_{t\searrow 0}  (S_t v)^2$ $= v^2$, where the last equality follows from
    $ \lim_{t\searrow 0}  S_t v= v.$\\
    \end{enumerate}
\end{rem}

If $E$ is a Lusin topological space  and $\mathcal{B}=\cb(E)$ is the Borel $\sigma$-algebra, then
a family $\Phi=(\Phi_t)_{t\geqslant 0}$ of mappings on $E$ is called {\it right continuous flow} (cf.  \cite{Sh88}, page 41) provided that
$(sd1) - (sd3)$ hold and 
in  addition:

\medskip
$(sd4')$ For each $x\in E$ the function $t\longmapsto \Phi_t(x)$ is right continuous on $[0, \infty)$.

\medskip

Clearly, any right continuous flow is a semi-dynamical system, because $(sd4')$ implies $(sd4)$, by taking 
$\cC_o$ a countable subset of  $bpC(E)$ which separates the points of  $E$. 
If the function $t\longmapsto \Phi_t(x)$ is continuous on $[0, \infty)$ for all $x\in E$ then 
$\Phi$ is called {\it continuous flow}. 

\begin{rem}
One may regard a right continuous flow $\Phi=(\Phi_t)_{t\geqslant 0}$ as a {\it deterministic right Markov process} 
$X=(\Omega,\cf,\cf_t, X_t, \Pp^x)$ in the following way: 
$\Omega:=E$, $\cf=\cf_t:=\cb(E)$,  $X_t(x):=\Phi_t(x)$  for all $x\in \Omega$ and $t\geqslant 0$, and $\Pp^x:=\delta_x$.
\end{rem}

Let $\cv=(V_\alpha)_{\alpha>0}$ be the resolvent of kernels associated with $\Ss$, $V_\alpha f =\int_0^\infty e^{-\alpha t} f(\Phi_t) \d t.$
We fix $\beta >0$, a strictly positive  function $f_o\in bp\cb(E)$,  and put $u_o := V_\beta f_o$.
We define now the {\it capacity} induced by $\phi$, by regarding $\phi$ as a (deterministic) right process. 
Let $\lambda$ be a finite measure on $E$ and consider the functional
$M\longmapsto c^\beta_\lambda (M),$ $M\subset E$, defined as
$$
c^\beta_\lambda (M) : = \inf \left\{ \int_E  e^{- \beta D_G } u_o(\Phi_{D_G })  \d \lambda  :  G \mbox{ open, }   M\subset G \right\}, 
$$
where $D_G$ is the {\it first  entry time}  of $G$,  $D_G(x):= \inf\{ t\geqslant 0: \Phi_t(x) \in G \},$  $x\in E.$
For measurability properties of the first entry and hitting times  in a set, for semi-dynamical systems with general state space see \cite{BezBu05}.
It turns out that $c^\beta_\lambda$ is Choquet capacity on $E$; see e.g. \cite{BeBo04} and also \cite{BeRo11} and \cite{BeBo05}.
Recall that the capacity $c^\beta_\lambda$ is called {\it tight} provided that there exists an increasing  sequence $(K_n)_n$  of compact sets such that
$\inf_n c^\beta_\lambda (K_n)=0$.

We can state now the first main result, which shows that every semi-dynamical system becomes a continuous flow with respect to a convenient Lusin topology.

\vspace{1mm}

\begin{thm} \label{thm2.2}  % Theorem 2.2
Let $\Phi=(\Phi_t)_{t\geqslant 0}$ be a semi-dynamical system on a Lusin measurable space $(E, \cb)$.
Then there exists a Luzin topology $\ct$ on $E$ such that $\cb=\cb(E)$ is the  Borel $\sigma$-algebra and $\Phi$ is a continuous flow with respect to this topology, 
such that  the map $x\longmapsto \Phi_t(x)$ is continuous on $E$ for all $t\geqslant 0$. 
For every finite measure $\lambda$ on $E$ and $\beta >0$ the capacity $c^\beta_\lambda$ is tight.
\end{thm}

\begin{proof}
Since by $(sd3)$ we have $\lim_{\alpha \to \infty} \alpha V_\alpha f= f$ pointwise on $E$ for all $f\in \mathcal{C}_o$, it follows that $\mathcal{E}(\cv_\beta)$ generates $\cb(E)$, where  $\beta >0$.
In addition, if $u, v\in \mathcal{E}(\cv_\beta)$ then $u\wedge v:=\inf(u, v)$ also belongs to $\mathcal{E}(\cv_\beta)$, so, all the points of $E$ are non-branch points with respect to
${\mathcal V}_\beta$.

The required Lusin topology $\ct$ is going to be generated by a convex cone of bounded $\cv_\beta$-excessive functions $\calr$, called a {\it Ray cone}.
Let us recall its usual construction, as, e.g.,  in \cite{BeRo11}, the proof of Proposition 2.2:
Let 
$\ds\mathcal{R}_0:=V_{\beta}(\cC_o)\cup {\Q}_{+}$. 
The Ray cone $\mathcal{R}$ is given by the closure in the $\sup$ norm
of
$\bigcup_{n\geqslant 0}\mathcal{R}_n$, where $\mathcal{R}_n$ 
is defined inductively as follows:\\
\centerline{$\mathcal{R}_{n+1}:= {\Q}_{+} \cdot \mathcal{R}_n
\cup (\sum_{f}\mathcal{R}_n) \cup (\bigwedge_{f}\mathcal{R}_n)
\cup (\cup_{\alpha \in {\Q}^{*}_{+}} V_{\alpha}(\mathcal{R}_n))
\cup (\cup_{t \in {\Q}^{*}_{+}} S_t(\calr_n))\cup
V_{\beta}((\mathcal{R}_n-\mathcal{R}_n)_{+})$,}\\
where $\bigwedge_f \calr_n$  is the set of all functions of the form $u_1\wedge u_2\wedge \cdots \wedge u_k$
with $u_i\in \calr_n ,$ $i\leqslant  k$,   
and ${\sum}_f \Q_+\!\! \cdot\! \calr_n$ 
is the set of all functions of the form $q_1u_1+q_2u_2+ \cdots +q_ku_k$  with $q_i\in \Q_+$.

Note that $\calr$ generates $\cb(E)$ which is thus the  Borel $\sigma$-algebra of $\ct$.
Since $t\longmapsto S_t V_\alpha f(x)$ is continuous and
$S_t(u\wedge v)=S_tu \wedge S_t v$, it follows inductively  that 
$t\longmapsto S_t u(x) $ is continuous on $[0, \infty)$ for all $x\in E$ and $u\in \bigcup_{n\geqslant 0}\mathcal{R}_n$,  and therefore for all $u\in \calr$.
Hence $t\longmapsto u(\Phi_t(x))$ is continuous on $[0, \infty)$ for all $u\in \calr$, that is, $\Phi$ is a $\ct$-continuous  flow.

We have $S_t (\calr) \subset \calr$ for all $t\in\Q_+$. So, clearly, $S_t u$ is $\ct$-continuous on $E$ if  $t\in\Q_+$ and
therefore  $x\longmapsto \Phi_t(x)$ is   $\ct$-continuous on $E$ for all $t\in\Q_+$.
Because for all $u\in \calr$  the function $t\longmapsto S_t u(x) $ is decreasing, it follows that
$S_t u=\sup_{\Q_+\ni t_n\searrow t} S_{t_n} u = \inf_{\Q_+\ni t_n\nearrow t} S_{t_n} u$ and thus the function $S_t u$ is $\ct$-continuous on $E$ for all $t> 0$.
We conclude that  $x\longmapsto \Phi_t(x)$ is $\ct$-continuous on $E$ for all $t\geqslant 0$.

According to  \cite{LyRo92} and \cite{BeBo05} (see also \cite{MaRo92},  \cite{BeRo11a}, and \cite{BeTr11}), 
the tightness property of the capacity $c^\beta_\lambda$ is a direct consequence of the  continuity of the trajectories of $\Phi$ in the topology $\ct.$
\end{proof}

\begin{rem} \label{rem2.3}  % Remark 2.3
$(i)$ The Lusin  topology from the above theorem is actually a {\rm Ray topology} with respect to the resolvent  $(V_\alpha)_{\alpha >0}$ of  $\Ss$;  for details see e.g. \cite{BeBo05} and \cite{BeBoRo06a}.

 $(ii)$ Theorem \ref{thm2.2} extends a result about right continuous flows from \cite{Sh88}, (47.8) at page 220.
\end{rem}

\section{The extended weak generator}\label{sec3}  % Section 3
In this section we extend to unbounded real-valued functions  
 the classical weak generator acting on bounded functions, considered by E.B. Dynkin (cf. \cite{Dy65}   pag. 55);
see also  \cite{Fi88} and \cite{Li11}. 
Notice that an extended generator was considered in \cite{Bou81} (and the references therein), however, only for bounded functions in the domain of the operator.
Also, we shall complete the approach from  \cite{HiYo12}.

Let  $\Tt =(T_t)_{t\geqslant 0}$ be a sub-Markovian transition function with induced  resolvent $\cu=(U_\alpha)_{\alpha > 0}$, and set
\begin{equation}\label{eq:B0}
\cb^0 = \cb^0(\Tt) \! := \{ f \in \!{[\cb]}:   T_t(|f|)< \infty \mbox{ for all } t>0 \mbox{ and } f = \! \lim_{s\searrow 0} T_s f \mbox{ pointwise on } E \}    
\end{equation}

\vspace{1.5mm}

Clearly, we have $[\ce_\alpha] \subset \cb^0=\cb^0(\Tt_\alpha)$ for every $\alpha\geqslant 0$.  
If  $\Tt =(T_t)_{t\geqslant 0}$ is the transition function of a right Markov process with (Lusin topological) state space $E$, 
then every bounded finely continuous function belongs to $\cb^0$, in particular,  $bC(E)\subset \cb^0$. 

Define also
\begin{align}
\cb_e&:= \{ f \in [\cb]: \exists\; h\in \ce \mbox{ with } |f|\leqslant h  \mbox{ and } f = \! \lim_{s\searrow 0} T_s f \mbox{ pointwise on } E  \}, \\
\cb_o &= \cb_o(\Tt)\\
&:= \{ f \in \cb^0: \forall  \alpha >0 \;\exists\;  t_o > 0, h_\alpha \in p\cb \mbox{ such that } \sup_{0<s< t_o}\! T_s | f |  \leqslant h_\alpha  \mbox{ and }  U_\alpha h_\alpha < \infty  \} \nonumber\\
\cb_{oo}&=  \cb_{oo} (\Tt)\\
&:= \{ f\in \cb_o:  \forall t>0 \;\exists \;   t_o > 0, h_t\in p\cb \mbox{ such that } 
\sup_{0<s< t_o}\!   T_s | f | \leqslant h_t \mbox{ and }  T_t h_t < \infty   \}.\nonumber
\end{align}
Several properties of the sets $\cb^0$, $\cb_e$, $\cb_o$,  and $\cb_{oo}$  are collected in the following lemma, whose proof is included in Appendix {(A.1)}.

\begin{lem} \label{lem3.1} % Lemma 3.1
The following assertions hold.  

\begin{enumerate}
    \item[(i)] For each $\alpha >0$ one has  $U_\alpha(\cb_o) \subset \cb_{oo}$ and if $t>0$ then $T_t (\cb_{oo})\subset   \cb_{oo}$. 
    If $\beta>0$ then  $ \cb_o =   \cb_o(\Tt_\beta)$ and   $\cb_{oo} =  \cb_{oo} (\Tt_\beta)$. 
    \item[(ii)] If $\alpha, t >0$ then  $U_\alpha(\cb_e) \subset \cb_{e}\subset \cb_{oo}$ and $T_t (\cb_{e})\subset   \cb_{e}$.
    \item[(iii)] We have {$b\cb_o= b\cb_{oo}=b\cb_e= b\cb^0$.}
    \item[(iv)] We have $[\cE]\cup b[\cE_\alpha] \subset \cb_{oo}$, $\alpha >0$. 
    If $f\in [\cb]$ is such that $U(|f|) <\infty$ then $Uf\in \cb_{oo}$.
\end{enumerate}
\end{lem}

\begin{cor} \label{cor2.2} % Corollary 3.2
If $\Tt =(T_t)_{t\geqslant 0}$ is the transition function of a right Markov process with Lusin topological state space $E$, and $f\in C(E)$ is such that there exists $h\in \ce$ with $|f|\leqslant h$, then $f\in \cb_{e}$. 
In particular, $bC(E)\subset \cb_{e}$.
\end{cor} 
 
Further, let us consider 
\begin{equation}
\begin{split}
    \cd(L):= \biggl\{u\in \cb_o : \forall \alpha >0\; \exists t_o>0,h_\alpha \in p\cb &\mbox{ with } \sup_{0<t<t_o} \left|\frac{T_tu- u}{t }\right| \leq h_\alpha, U_\alpha h_\alpha<\infty, \\
    &\mbox{ and } \lim_{t\searrow 0} \frac{T_tu- u}{t}\in \cb_o \mbox{ pointwise on } E \biggr\} 
\end{split}
\end{equation}
Clearly, $\cb_o$, $\cb_{oo}$, $\cb_e$, and $\cd (L)$  are  vector spaces 
and define the linear operator  
\begin{equation}
L: \cd(L)\rightarrow  \cb_o, \quad  Lu(x):=  \lim_{t\searrow 0} \frac{T_tu (x)- u (x)}{t},\;f\in \cd(L), \;x\in E.   
\end{equation}
Define  also  
\begin{equation}
    \cd_o(L)  := \{ u\in \cd(L): Lu\in \cb_{oo} \} \quad  \mbox{ and } \quad \cd_e(L) := \{ u\in \cd(L)\cap \cb_e: Lu\in \cb_{e} \}.
\end{equation}
The operator  $(L, {\mathcal{D}}(L))$ is called the {\it extended weak generator} of $\Tt= (T_t)_{t\geqslant 0}$.

\begin{rem}
\begin{enumerate}
    \item[(i)] Recall the definition of the {\rm weak generator}   $(L_w, \cd(L_w))$ considered in \cite{Dy65}: $\cd(L_w)$ is the set of all bounded functions $ f\in \cb^0$ such that  $\left( \frac{T_tf(x)- f(x)}{t}\right)_{t, x}$ is bounded for $x\in E$ and  $t$    in a neighbourhood of  zero, there exists $\lim_{t\searrow 0} \frac{T_tf - f}{t}$   pointwise and the above limit is an element of $\cb^0$. 
    If $\alpha >0$ then   $\cd(L_w)=U_{\alpha}(b\cb^0)$, it is independent of ${\alpha} >0$ and if $u= U_{\alpha} f$ with $f\in \cb^0$, then $({\alpha} -L_w) u= f$.
    \item[(ii)] In \cite{HiYo12} an  {\it extended generator}  $(\overline{L}, {\mathcal{D}}(\overline{L} ))$  of $\Tt= (T_t)_{t\geqslant 0}$ was considered by taking into account unbounded real-valued functions also, as follows: Let $u, g \in \cb^0,$  then $u$ belongs to the domain  $\mathcal{D} (\overline{L})$ of $\overline{L}$ and $g=\overline{L} u$ provided that 
    \begin{equation} \label{ext-generat}
    \forall \; t > 0, x\in  E \mbox{  we have }  \! \int_0^t \!\! T_s(|g|) (x) \d s < \infty  \mbox{ and }  T_t u(x) =  u(x) + \!  \displaystyle\int_0^t \!\! T_s g (x) \d s.
    \end{equation}
    \item[(iii)] Assume that $\Tt= (T_t)_{t\geqslant 0}$ is the transition function of a right Markov process $X= (\Omega, \cf_t, X_t,  \Pp^x)$  with Lusin topological state space $E$. 
    According to \cite{HiYo12}, Proposition 4.1 (see also \cite{Fi88}, page 354, the proof of Theorem (4.1)), we have the following equivalent definition for the extended generator: If  $u, g\in \cb^0$  then $u \in \cd(\overline{L})$   and $\overline{L} u=g$  if and only if for all $x\in E$ we have {$ \displaystyle \int_0^t \!\! T_s(|g|) (x) \d s < \infty$ for all $t > 0$  and  $\left( u(X_t)- uX_0)- \displaystyle \int_0^t g(X_s) \d s \right)_{t\geqslant 0} $ is a $(\cf_t)$-martingale under $\Pp^x$.}
\end{enumerate}    
\end{rem}

The next result collects  properties of the extended weak generator.
Several arguments used in the proof 
are similar to the case of the 
$C_0$-semigroups of contractions on a Banach space of functions; see, e.g., \cite{EthKu86}, Ch. 1, section 2. 
In particular, assertion $(viii)$ below is a pointwise version of Theorem 1.3 from \cite{Dy65}, Ch. I, section 3.
For the reader convenience we present its proof in Appendix (A.2).

\begin{prop} \label{prop3.3} % Proposition 3.3
The following assertions hold for a sub-Markovian transition function  $\Tt =(T_t)_{t\geqslant 0}$, 
its  resolvent $\cu=(U_\alpha)_{\alpha > 0}$, and the extended weak generator  $(L, {\mathcal{D}}(L))$.

\begin{enumerate}
    \item[(i)] If $\alpha >0$ then   $\cd(L)=U_{\alpha}(\cb_o)$ and  it is independent of ${\alpha} >0$. 
    If $f\in \mathcal{B}^0(\Tt)$, $\alpha \geqslant 0$ and $u=U_\alpha f$ then $({\alpha} - {L}) u= f$. 
    If $f\in b\cb^0$, $t>0$, and $u=\int_0^t T_sf \d s$ then $u\in \cd(L)$ and $Lu=  T_t f -f$.
    \item[(ii)] The operator $(\overline{L}, {\mathcal{D}}(\overline{L}))$ is well defined and we have $\overline{L}u (x)= \lim_{t\searrow 0} \frac{T_tu  (x)- u (x)}{t},$  $x\in E$,  $u\in \cd(\overline{L})$.
    \item[(iii)] We have $\cd(L_w)\subset  \cd_e(L) \subset \cd(L)= \{u \in \cd (\overline{L})\cap \cb_{o}: \overline{L} u \in \cb_o\} \subset \cb_{oo}$, $\overline{L}|_{\cd(L)}= L$, and $L|_{\cd(L_w)} = L_w$.
    \item[(iv)] One has $\cd_o(L)= U_{\alpha}(\cb_{oo})$ for each $\alpha >0$. If  $t >0$ then $T_t(\cd_o(L))\subset \cd_o(L)$, $T_t(\cd(\overline{L}))\subset \cd(\overline{L})$, $\overline{L}\circ T_t = T_t \circ \overline{L}$ on $\cd(\overline{L})$, and $L\circ T_t = T_t \circ L$ on $\cd_o(L)$.
    \item[(v)] If $\beta >0$ and $(L^\beta , \cd(L^\beta ) )$ (resp. $(\overline{L^\beta} , \cd(\overline{L^\beta} ) )$ denotes the extended weak generator  (resp. the extended generator) of the  transition function  $\Tt_\beta$, then  $\cd (L)\subset \cd (L^\beta )$ (resp. $\cd (\overline L)\subset \cd (\overline{L^\beta} )$), $L^\beta  u = Lu -\beta u$ for every $u\in \cd(L)$ (resp. $\overline{L^\beta}  u = \overline Lu -\beta u$ for every $u\in \cd(\overline L)$), and $\cd_o(L) =  \cd_o(L^\beta )$.
    \item[(vi)] We have $\cd_e(L)= U_\alpha (\cb_e)\subset \cd_o(L)$ for each $\alpha >0$ and if  $t >0$ then $T_t(\cd_e(L))\subset \cd_e(L)$.
    \item[(vii)] Let 
$
\cd^c_o(L):= \{ u\in \cd_o(L):  [0, \infty)\ni  t\longmapsto LT_t u (x)  \mbox{ is continuous for each } x \in E\}.
$ 
 If $t, \alpha >0$ then $T_t(\cd^c_o(L))\subset \cd^c_o(L)$  and 
 $U_\alpha (\cd(L)) \subset \cd^c_o(L)$.
 If $\beta >0$ then
 $U_\beta U_\alpha ({b[\cb]})\subset  \cd^c_o(L)$.
    \item[(viii)] If $u\in b\cd^c_o(L)$  then $[0, \infty) \ni t\longmapsto T_tu(x)$ is continuously differentiable for each $x\in E$ and $(T_tu(x))' = LT_t u(x)$. 
    Moreover,  $u_t := T_t u$, $t\geqslant 0$,  is the unique solution of the equation
    \begin{equation} \label{eq3.3}
    \frac{\d u_t}{\hspace{-1.5mm}\d t} = L u_t, t\geqslant 0, 
    \end{equation}
    such that $u_0=u$, $u_t \in \cd_o(L)$, $\| u_t \|_\infty$ is bounded, $Lu_t\in \cb_{oo}$, and  $[0, \infty )\ni t\longmapsto L u_t(x)$ is continuous  for all $x\in E$.
\end{enumerate}
\end{prop}

\begin{cor} \label{cor3.4} % Corollary 3.4
Assume that  $\Tt =(T_t)_{t\geqslant 0}$ is the transition function of a right Markov process $X=(\Omega, \cf, \cf_t, X_t, \Pp^x)$ with Lusin topological state space $E$.
Then the following assertions hold.
\begin{enumerate}
    \item[(i)] If $f\in C(E)$ is such that there exists $h\in \ce$ with $|f|\leqslant h$, then $U_\alpha f \in \cd_e(L)$ for each $\alpha >0$. In particular, $U_\alpha (bC(E) )\subset \cd_{e}(L)$. 
    The above assertions are still  true if we replace the continuity condition by the weaker one of fine continuity.
    \item[(ii)] If $\Tt =(T_t)_{t\geqslant 0}$ is a {\rm Feller semigroup}, i.e., each kernel $T_t$, $t>0$, leaves invariant $bC(E)$, then $U_\alpha (bC(E) )\subset \cd_o^c (L)$.
    \item[(iii)] The martingale problem associated with $(\overline{L}, \cd (\overline{L}))$ has a solution. 
    More precisely, for every $u\in \cd (\overline{L})$ and  $x \in E$, the process
    $$
    \left(u(X_t) -u(X_0) -\int_0^t \overline{L}u(X_s) \d s\right)_{t\geqslant 0}
    $$
    is a $(\cf_t)$-martingale under $\Pp^x$.
\end{enumerate}
\end{cor} 

\vspace{2mm}

The following additional property of  $\Ss =(S_t)_{t\geqslant 0}$ will be considered further on:
\begin{equation}\label{3.4}
 \exists\;\cC_o\subset bp\cb  \mbox{ such that  } 1\in \cC_o,  \cC_o    \mbox{ generates } \cb,  \mbox{ and }  \lim_{t\searrow 0} S_t f (x)= f(x) \mbox{  for all } x\in E.
\end{equation}

\begin{rem}
As a consequence of $(\ref{3.4})$ we have for all $\alpha, \beta >0$:
\begin{equation}\label{3.5}
\mbox{If }  (\ref{3.4}) \mbox{ holds then } \sigma(V_\beta V_\alpha(b\cb_{oo}))= \sigma(V_\alpha(b\cb_{oo}))=\cb,
 \end{equation}
where  $\cv=(V_\alpha)_{\alpha >0}$ is the resolvent of $\Ss$ and $\cb_{oo}=\cb_{oo}(\Ss)$.
Indeed, by $(\ref{3.4})$ if follows that for every $f\in \cC_o$ we have pointwise $lim_{\alpha\to \infty} \alpha V_\alpha f= f$ and therefore 
$\cC_o\subset bp\sigma (V_\alpha (bp\cb))$, hence $\cb=\sigma (\cC_o)\subset \sigma (V_\alpha (bp\cb))$, so, $\sigma (V_\alpha (b[\cb]))=\cb$. 
On the other hand by Lemma \ref{lem3.1} $(iv)$  we have $V_\alpha (b[\cb])\subset b\cb_{oo}\subset  b[\cb]$ and therefore
 $\sigma(b\cb_{oo})=\cb$. 
 By Lemma \ref{lem3.1} $(i)$  $V_\alpha (b\cb_{oo})\subset \cb_{oo}$ and therefore the vector space $V_\alpha (b\cb_{oo})$ 
 does not depend on $\alpha >0$ and $\sigma(V_\alpha (b\cb_{oo})) \subset \cb_{oo}$. 
 The converse inclusion also holds because for every $f\in b\cb_o$ we have
 $\lim_{\alpha\to \infty} \alpha V_\alpha f=f$ pointwise on $E$ and we conclude that the last equality from $(\ref{3.4})$ is proven.
  Observe that  the resolvent equation  implies that the vector space 
 $V_\beta V_\alpha ({b\cb_{oo}})$ also does not depend on $\alpha$ and $\beta$.
 If $f\in b\cb_o$ then $lim_{\beta \to \infty} \beta V_{\beta} V_\alpha f=  V_\alpha f$ pointwise on $E$, hence
 $V_\alpha f\in b\sigma(V_\beta V_\alpha(b\cb_{oo}))$ and therefore $\sigma(V_\alpha(b\cb_{oo}))\subset \sigma(V_\beta V_\alpha(b\cb_{oo}))$ and 
so, the first equality is also clear.    
\end{rem}

\paragraph{Non-autonomous semi-dynamical systems.} Let $(E, \cb)$ be a Lusin measurable space
and let $\Phi=(\Phi_{s,t})_{t\geq s\geq 0}$ be a family of mappings $\Phi_{s,t}:E\rightarrow E, t\geq s\geq 0$.
Inspired by e.g. \cite{Kunita93}, we say that $\Phi$ is a {\it non-autonomous semi-dynamical system} on $E$ provided that the  following conditions are satisfied:
\begin{enumerate}
    \item[(Nsd1)] $\Phi_{s,t}(x)= \Phi_{r,t}(\Phi_{s,r}(x))$  for all $t\geq r\geq s \geq 0$ and $x\in E$;
    \item[(Nsd2)] $\Phi_{s,s}(x)=x$    for all  $s\geq 0$, $x\in E$;
    \item[(Nsd3)] For each $t>0$ the function $ [0,\infty)\times E \ni (s,x) \longmapsto \Phi_{s,s+t}(x)$ is measurable;
    \item[(Nsd4)] There exists a countable set $\cC_o\subset bp\cb$ such that $\cC_o$ separates the points of $E$ and 
$\lim_{t\searrow 0} f(\Phi_{s,s+t}(x))= f(x)$ for all $s\geq 0, x\in E$ and $f\in \cC_o$.
\end{enumerate}
The paths of unique strong solutions to Ito SDEs on $\mathbb{R}^d$ which depend continuously on the initial data are typical examples of such non-autonomous semi-dynamical systems (see e.g. \cite{Fl10}).

Given a $\Phi$ as above, it is a straightforward to check that $\overline{\Phi}:=\left(\overline{\Phi}_{t}\right)_{t\geq 0}$ defined by
\begin{equation*}
\overline{\Phi}_{t}:[0,\infty)\times E \rightarrow [0,\infty)\times E, \quad \overline{\Phi}_{t}(s,x):=(s+t,\Phi_{s,s+t}(x)), \quad t,s\geq 0, x\in E,
\end{equation*}
is a semi-dynamical system on $[0,\infty)\times E$.

Thus, the results obtained in this work for (autonomous) semi-dynamical systems can be easily reinterpreted for non-autonomous semi-dynamical systems.

\subsection{The extended weak generator of a semi-dynamical system} \label{subsec3.1}

We have the following characterization of those Markovian transition functions that correspond to semi-dynamical systems:
\begin{prop} \label{prop3.5} % Proposition 3.5
Let $\Ss= (S_t)_{t\geqslant 0}$ be a Markovian transition function on $(E, \cb)$ and $(D, \cd(D))$ be its extended weak generator.
Then the following assertions are equivalent.
\begin{enumerate}
    \item[(i)] $\Ss=(S_t)_{t\geqslant 0}$ is the {transition function of } a  semi-dynamical system on $E$.
    \item[(ii)] The transition function $\Ss=(S_t)_{t\geqslant 0}$ satisfies  $(\ref{3.4})$ and it is multiplicative, that is, for every $f,g\in bp\cb$ and $t>0$ we have $S_t(f g)=(S_tf) (S_tg)$.
    \item[(iii)] $\Ss= (S_t)_{t\geqslant0}$   satisfies  $(\ref{3.4})$, $\cb_e$ and $\cd_e (D)$  are   algebras, $\Ss=(S_t)_{t\geqslant 0}$ is multiplicative on $\cb_e$, and if  $u\in \cd_e(D)$ {then }  $Du^2=2uDu$.
    \item[(iv)] $\Ss= (S_t)_{t\geqslant0}$ satisfies  $(\ref{3.4})$, $\cd_b^c(D) :=\{ u\in b\cd^c_o(D): Du\in b\cb_{oo}\}$ is an  algebra, and if $u\in \cd_b^c(D)$ then  $Du^2=2uDu$.
    \item[(v)] $\Ss= (S_t)_{t\geqslant0}$ satisfies  $(\ref{3.4})$ and there exists an algebra $\ca\subset  \cd_b^c(D)$ which generates $\cb$, $S_t u \in \ca,$ $t> 0$, and $Du^2=2uDu$ for each $u\in \ca$.
\end{enumerate} 
\end{prop}
 
\begin{proof}
The implication $(i)\rightarrow (ii)$ is clear;
notice that $(sd4)$ implies that  $(\ref{3.4})$ holds.

\medskip
$(ii)\rightarrow (iii)$.
We show first that 
\begin{equation}\label{2.5}
 \mbox{ if }\  \Ss=(S_t)_{t\geqslant 0} \mbox{ is multiplicative and } v \in \ce_\beta   \mbox{ then } v^2 \in \ce_{2\beta},  
\end{equation}
where $\beta \geqslant 0$ and  $\ce_0:=\ce$.
Indeed, since $\Ss=(S_t)_{t\geqslant 0}$ is multiplicative
we have $e^{-2\beta t} S_t (v^2)= (e^{-\beta t} S_t v)^2\leqslant v^2$, where the inequlity holds because $v\in \ce_\beta$.
Then clearly $\lim_{t\searrow 0}  e^{-2\beta t} S_t (v^2)=  \lim_{t\searrow 0}  (S_t v)^2= v^2$, where the last equality follows from
$ \lim_{t\searrow 0}  S_t v= v.$

As a consequence of $(\ref{2.5})$ we  have:
\begin{equation}\label{2.6}
 \mbox{ if } \Ss= (S_t)_{t\geqslant 0} \mbox{ is multiplicative then }  \cb_e \mbox{ is an algebra, i.e.,  if } f \in \cb_e  \mbox{ then } f^2\in \cb_e.  
\end{equation}
Indeed, if $f\in \cb_e$ and $|f|\leqslant h\in \ce$ then by $(\ref{2.5})$  we get $f^2\leqslant h^2\in \ce$ and 
because $S_s (f^2) = (S_s f)^2$   we also have $\lim_{s\searrow 0} S_s (f^2) = (\lim_{s\searrow 0} S_s f)^2= f^2$.
So, by Lemma  \ref{lem3.1} $(ii)$ we conclude that $f^2$ also belongs to $\cb_e$.

Let now $u\in \cd_e(D)$, $|u|\leqslant h\in \ce$. 
By $(\ref{2.6})$ we get $u^2\in \cb_e$  and
$S_tu^2 - u^2 =  (S_t u - u(x)) (S_t u  + u)$, $t>0$.
We have $|S_tu+ u|\leqslant 2h$ and 
$ \sup_{0<t<t_o} |\frac{S_tu- u}{t }| \leqslant h_\alpha$, with $V_\alpha h_\alpha<\infty$ on $E$, where
$\cv=(V_\alpha)_{\alpha >0}$ is the resolvent of $\Ss$.
Consequently,
$ \sup_{0<t<t_o} |\frac{S_tu^2- u^2 }{t }| \leqslant 2h_\alpha h$ and
we have $V_\alpha (h_\alpha h) = \int_0^\infty  e^{-\alpha s} (S_s h_\alpha) S_s h \leqslant h V_\alpha h_\alpha  <\infty$ on $E$.
Because  
$\lim_{t\searrow 0} S_t u=u$ pointwise on $E$, we conclude that for every $x\in E$ there exists the limit
$\lim_{t\searrow 0}\frac{S_tu^2 (x) - u^2(x)}{t} = \lim_{t\searrow 0} \frac{S_t u(x) - u(x)}{t} \lim_{t\searrow 0} (S_t u (x) + u(x) )= Du(x) 2u(x).$
So,  $u^2\in \cd(D)\cap \cb_e$ and $D u^2 = 2u Du$.
Moreover, $u$ and $Du$ both belong to $\cb_e$, therefore $(\ref{2.6})$ implies that  $Du^2 \in \cb_e$, hence $u^2\in \cd_e(D)$.

\medskip
$(iii)\rightarrow (iv)$ Notice  first that  $\cd^c_b (D)\subset \cd_e(D)$,  because $1\in \ce$.
If  $u\in \cd^c_b(D)$ then by the hypothesis $(iii)$ we have $u^2 \in \cd_e(D)$ and $D u^2 = 2u Du$.
In addition, $u, Du\in b\cb_{oo}$, hence $Du^2$ also belongs to $b\cb_e$ which is an algebra included in $b\cb_{oo}$.
Consequently, $u^2\in \cd_o(D)$.
Since $DS_t (u^2)= 2(S_t u)(DS_t u)$ and the functions  $S_t u$ and $DS_t u$ are continuous in $t$, 
we conclude that $u^2$ also belongs to $\cd^c_b(D)$.  

\medskip
$(iv)\rightarrow (v)$
Assume that $(iv)$ holds, then $\cd^c_b(D)$ is multiplicative and we show that it generates $\cb$.
Indeed, by  Proposition \ref{prop3.3} $(vii)$ we have $V_\beta V_\alpha ({b\cb_o})\subset \cd^c_b(D)$.
From 
$(\ref{3.5})$ we get $\cb= \sigma(V_\beta V_\alpha ({b\cb_o}))\subset \sigma(\cd^c_b (D))$ and thus  $\sigma(b\cd^c_b(D))=\cb$.

\medskip
$(v)\rightarrow (ii)$. 
Let now $u\in \ca$ as in $(v)$.
If we put $v_t:= (S_t u)^2$ then by hypothesis we have $v_t\in \ca$, $t\geqslant 0$, $\sup_{0\leqslant t<\infty} \| v_t \|_\infty\leqslant \| u\|_\infty^2$, and $t\longmapsto Dv_t(x)$ is continuous for each $x\in E$.
Using  Proposition \ref{prop3.3} $(viii)$ we obtain $\frac{\d v_t}{\hspace{-1.5mm}\d t} = 2S_t u \cdot D S_t u= Dv_t,$ $t\geqslant 0$, 
with  $v_0= u^2$. 
 By the uniqueness property of  the equation $(\ref{eq3.3})$   it follows that 
 $(S_t u)^2 = S_tu^2,$ 
 hence
 $(S_t u)(S_t v)= S_t (uv)$ for all $u, v\in \ca $. 
 Applying  \Cref{rem2.1}, (iii), we conclude that 
  $\Ss=(S_t)_{t\geqslant 0}$ is multiplicative and therefor assertion $(ii)$ holds.

  \medskip
 $(ii)\rightarrow (i)$. The proof of this implication is straightforward, however, for the reader convenience we give some details here.
 Let $S_t(x, \cdot)$ be the probability on $E$ induced by the Markovian kernel $S_t$ and $x\in E$, 
 $S_t(x, A):= S_t (1_A) (x)$ for all $A\in \cb$.
 Taking $f=g=1_A$ in the property of $\Ss=(S_t)_{t\geqslant 0}$ to be multiplicative we get $S_t(1_A)= (S_t(1_A))^2$ and therefore the number
 $S_t(x, A)$ should be either $0$ or $1$. 
 Hence $S_t(x, \cdot)$ is a Dirac measure on $E$, concentrated at a point  $\Phi_t(x)\in E$, $S_t(x, \cdot)=\delta_{\Phi_t(x)}$.
 We obtain  $S_t f(x)= f(\Phi_t(x))$ for all $f\in p\cb$, $x\in E$, and $t\geqslant 0$, and it is easy to check now that
 $\Phi=(\Phi_t)_{t\geqslant 0}$ verifies $(sd1)-(sd3)$, while $(sd4)$ follows from $(\ref{eq2.1})$.
 So, $\Phi=(\Phi_t)_{t\geqslant 0}$ is a semi-dynamical system on $E$ and $\Ss=(S_t)_{t\geqslant 0}$ is its transition function.
  \end{proof}

The following result concerns the algebraic structure of the extended generator %$(\overline D, \cd(\overline D)$ 
of a semi-dynamical system; its proof is deferred to Appendix (A.3).
\begin{prop}  \label{prop3.6} % Proposition 3.6
Let $\Ss= (S_t)_{t\geqslant 0}$ be the  transition function of  a  semi-dynamical system  on $(E, \cb)$ and let
 $(\overline D, \cd(\overline D))$ be its extended generator.
If $f\in \cd(\overline D)$  and $\int_0^t  S_s (| f\overline D f|)  \d s < \infty$ for all $t>0$ 
then 
$f^2 \in \cd (\overline D)$ and $\overline D f^2= 2 f \overline D f.$
In particular, $b\cd (\overline D)$ is an algebra.
\end{prop}

\paragraph{Example: The classical case of an Euclidean gradient flow.}
Let ${\bf B}:\mathbb{R}^d \rightarrow \mathbb{R}^d$ be a continuous vector field such that:
\begin{enumerate}
    \item[(B.i)] For each $r>0$ there exists a constant $c(r)$ such that for all $x,y\in\mathbb{R}^d,|x|,|y|\leq r$
    \begin{equation*}
       \langle {\bf B}(x)-{\bf B}(y),x-y\rangle \leq c(r)|x-y|^2 \qquad  (local \; weak \; monotonicity).
    \end{equation*}
    \item[(B.ii)] There exists a constant $c_0$ such that for all $x\in \mathbb{R}^d$
    \begin{equation*}
       \langle {\bf B}(x),x\rangle \leq c_0(1+|x|^2) \qquad  (weak \; coercivity).
    \end{equation*}
\end{enumerate}
Then, by e.g. [Rockner-Wei Liu], Therem 3.1.1 (applied for $\sigma \equiv 0$), for each $x\in \mathbb{R}^d$ there exists a unique solution $(\Phi_t(x))_{t\geq 0}\in C([0,\infty);\mathbb{R}^d)$ to the equation
\begin{equation}\label{eq:flow}
\left\{
\begin{array}{ll}
 \d \Phi_t(x)= {\bf B} (\Phi_t(x))\d t,  & t\geqslant 0,\\[2mm]
\Phi_0(x)=x.
\end{array}
\right.
\end{equation}
$(\Phi_t)_{t\geq 0}$ is a semi-dynamical system as considered in Section 2, which can be regarded as a (deterministic) right process with transition function $(S_t)_{t\geq 0}$, 
$$
S_tf(x)=f(\Phi_t(x)), \quad t\geq 0,x\in\mathbb{R}^d,f\in b\mathcal{B}(\mathbb{R}^d).
$$
Note that if $(D, \mathcal{D}(D))$ denotes the { weak} generator of the continuous flow  $\Phi= (\Phi_t)_{t\geqslant 0}$, then it is clear that
$$
Dv= {\bf B} \! \cdot \! \nabla v \, \mbox{ for all }  v\in C^1_b(\mathbb{R}^m).
$$

\subsection{Stopped  continuous flows} \label{subsec3.2}

In this subsection (more precisely, in Proposition \ref{prop3.8} below) we apply to continuous flows the classical technique of stopping a Markov process at its first entry time in a given set. 
This stopping technique has been used in \cite{BeIoLu-St21}, Remark 3.4, in studying   stochastic fragmentation processes for particles with spatial position on a surface.

Let  $\Phi=(\Phi_t)_{t\geqslant 0}$ be continuous flow on a Lusin topological space $E$
and let $\co$ be an open susbset of $E$.
Let $T$ be {\it the first entry time} in $\co^c=E\setminus \co$,
\begin{equation*}
T(x)= \inf \{ t\geqslant 0: \Phi_t(x) \in \co^c\}. 
\end{equation*}
The following properties are immediate:
\begin{enumerate}
    \item $T$ is a {\it terminal time}, that is, the mapping $E\ni x \longmapsto T(x)$ is  $\cb(E)$-measurable and
    $$
    t + T\circ \theta_t=T \mbox{ on } [t< T], 
    $$
    or equivalently, 
    $t+ T(\Phi_t (x))= T(x) \mbox{ if } t< T(x) \mbox{ for all } x\in E$.
    \item If $x \in \oco$ then $\Phi_{T(x)}  (x)\in \partial \co$.
    \item If $x \in {\oco}^{\, c}$ then $T(x)=0$, so, $\Phi_{T(x)} (x)=x$.
    \item We have $\Phi_T(x) (x)\in \co^c$ for every $x\in E$.
\end{enumerate}
For each $t\geqslant 0$ define the map 
$\Phi^o_t: E \rightarrow E$ as
\begin{equation*}
   \Phi^o_t (x):=
   \begin{cases}
        \Phi_t(x),\quad &t< T(x)\\
        \Phi_{T(x)}(x), \quad &t\geqslant T(x)
    \end{cases}
    , x\in E.
\end{equation*}

The announced result of this subsection is the following collection of statements, whose proofs are presented in Appendix (A.4).
\begin{prop}  \label{prop3.8} % Proposition 3.8
Then the following assertions hold.
\begin{enumerate}
    \item[(i)] The family $\Phi^o := (\Phi^o_t)_{t\geqslant 0}$ is a  continuous flow  on $E$ and it is called {\rm the stopped  flow} w.r.t. $T$. We have   $\oph_t(x)= \Phi_t(x)$  if $t< T(x)$ and $\Phi^o_t (x)=x$ for every $x\in  \co^c$ and $t\geqslant 0$.
    \item[(ii)] Let  $(D, \mathcal{D}(D))$  (resp.   $({D^o}, \mathcal{D}({D^o}))$) be  the extended { weak} generator of the continuous flow  $\Phi$  (resp. of the continuous flow  $\Phi^o$) on $E$. 
    We have ${D^o} u=0$ on $\co^c$ for all $ u \in \cd({D^o})$ and if in addition  $u\in \cd(D)$ then $D^o u= {D}u$ on $\co$.
    \item[(iii)] The set $\overline{\co}$ is  {\rm absorbing} for $\Phi^o := (\Phi^o_t)_{t\geqslant 0}$, that is, if $x\in \overline\co$ then  $\Phi^o_t (x) \in \overline\co$ for all $t\geqslant 0$.
    \item[(iv)] Define {\rm the restriction}  $\Phi^{\oco} = (\Phi^{\oco}_t)_{t\geqslant 0}$ of  $\Phi$ to $\oco$ as $\Phi^{\oco} (x) := \Phi^o_t(x)$ for all $x\in \oco$ and $t\geqslant 0$. 
    Then $\Phi^{\oco}$ is a continuous flow on $\oco$.
\end{enumerate} 
\end{prop}

\section{Continuous flow driving a Markov process}  \label{sec4} % Section 4

Let $(L, \cd(L))$ and $(D, \cd(D))$ be two extended weak generators on $E$. 
Define
\begin{equation*}
    \cd(DL):= \{ u\in \cd(D)\cap \cd(L): Lu\in \cd(D) \mbox{ and } DLu \in \cb^0(\Tt) \},
\end{equation*}
and $\cd(LD)$ analogously.

\vspace{1mm}

We can present now the  second main result of this paper.

\begin{thm} \label{thm4.1} %Theorem 4.1
Let $\Tt=(T_t)_{t\geqslant 0}$ be the transition function of a right (resp. Hunt) Markov process 
$X=(\Omega, \cf, \cf_t, X_t, \Pp^x)$ with state space $E$ and extended weak generator $(L, \cd(L))$.
Assume that
there exists a multiplicative set  $\cC_1\subset b\cC(E)$ which generates $\cb(E)$
such $T_t(\cC_1)\subset C(E)$
%and  $S_t T_{t}= T_{t}  S_t$ 
 for all $t >0$.
Let $\Phi=(\Phi_t)_{t\geqslant 0}$  be a  right continuous flow on $E$
such that the mapping $(x, t)\longmapsto \Phi_t (x)$ is continuous on $E \times [0, \infty)$, with transition function 
$\Ss=(S_t)_{t\geqslant 0}$ and extended weak generator $(D, \cd(D))$.
Suppose in addition that  {\rm $L$ and $D$ commute} in the sense that
\begin{equation*}
    \cd(DL) = \cd(LD)=: \cd_o \quad \mbox{ and } \quad DL= LD \mbox{ on }  \cd_o.
\end{equation*}
Furthermore, set
\begin{equation*}
    {X}_t^\Phi : = \Phi_t (X_t), \, t\geqslant 0.
\end{equation*}
Then the following assertions hold.
\begin{enumerate}
    \item[(i)] ${X}^\Phi :=(\Omega, \cf, \cf_t, {X}^\Phi_t, \Pp^x)$ is a right (resp. Hunt) Markov process with state space $E$ and the transition ${\Tt}^\Phi:= ({T}^\Phi_t)_{t\geqslant 0}$ defined as  ${T}_t^\Phi :=S_t T_t$ for all $t\geqslant 0$.
    \item[(ii)] Let $\cd_c:= U_\alpha V_\beta(bC(E)),$    $\alpha, \beta >0$. 
    Then $\cd_o\subset \cd_o(L)\cap \cd_o (D) \cap  \cd_o({L}^\Phi)$, $\cd_c \subset \cd_o^c(L)\cap \cd_o(D) \cap \cd (L^\Phi)$ and  
    $$
    {L}^\Phi= L+ D \mbox{ on } \cd_c.
    $$
\end{enumerate}
\end{thm}

\begin{proof}
$(i)$ We check first  that  ${X}^\Phi$   is a  (simple) Markov process with
${\Tt}^\Phi$ as transition function.
If  $f\in bp\cb$, $\mu$ is a probability on $E$, and $s, t\geqslant  0$ then by the Markov property of $X$ we obtain
$\Ee^\mu [ f (X^\Phi_{t+s} | \cf_t ] =  T_s( f(\Phi_{t+s} ) )(X_t)= T_sS_{t+s}f (X_t)= S_s T_s S_tf(X_t)= T_s^\Phi f(X^\Phi_t)$.
We have also
$T_{t-s}^\Phi f(X_s^\phi)= S_sT_{t-s}S_{t-s} f (X_s)= T_{t-s} S_t f (X_s)$ if $s<t$. 
It follows that for all $t\geqslant 0$ 
$ [s\longmapsto T_{t-s}^\Phi f(X_s^\phi)1_{[0, t)}$ is not right continuous$]$$=
[s\longmapsto T_{t-s}(S_t f) (X_s)1_{[0, t)}$ is not right continuous$]$ and by Corollary (7.9) from \cite{Sh88} we conclude that
${X}^\Phi$ is a right process.

$(ii)$ 
 Observe  that Corollary \ref{cor2.2} implies that $bC(E)\subset \cb_e(\Ss)\cap \cb_e(\Tt)\cap\cb_e(\Tt^\phi).$
If  $f\in bC(E)$,  because $T_t$ and $V_\beta$ commute, by dominate convergence we get 
$\lim_{t\searrow 0} S_t U_\alpha f= U_\alpha (\lim_{t\searrow 0} S_t  f)= f$.
Therefore, by Lemma \ref{lem3.1} $(ii)$ we deduce that $U_\alpha f \in \cb_{oo}(\Ss)$ and
consequently,  if $u=U_\alpha V_\beta f$ then $u\in \cd_o(D)$.
Analogously, $u$ belongs to $\cd_o (L)$ too.
In addition,
$LT_t u= \alpha T_t u- V_\beta T_t  f$,  $DS_t u= \beta S_t u- U_\beta S_t  f$
 and so, the functions $LT_t u(x)$ and $DS_t u(x)$, $x\in E$,   are continuous in $t$, hence $u\in \cd_o^c(L)\cap \cd_o^c(D)$.

Because $\lim_{t\searrow 0} \frac{T_t u - u}{t}= Lu$ if and only if  $ \lim_{t\searrow 0}  \frac{e^{-\alpha t} T_t u - u}{t}= Lu-\alpha u$, 
we may suppose that the potential kernels $U$ and $V$ are bounded and that 
$u=UVf$, hence $U|f|$ and $V|f|$ are bounded functions.
We have
$\frac{T_t^\Phi u- u}{t}  = S_t (\frac{T_t u - u}{t}) +  \frac{S_t u-u}{t}$ 
and so, to show that  $ u\in \cd (L^\Phi)$  and ${L}^\Phi u = Lu+ Du$,  it is sufficient to prove that
$\lim_{t\searrow 0} S_t( \frac{T_t u - u}{t})= -Vf  \ \mbox{ pointwise on } E.$
We have $T_t u-u = - V( \int_0^t T_s f \d s)$, 
$ S_t (\frac{T_t u - u}{t})=  -V  (S_t \frac 1 t \int_0^t T_s f \d s)=
-Vf - V (S_t \frac 1 t \int_0^t T_s g \d s -f).
$
Therefore, it remains to show that $\lim_{t\searrow 0} V (S_t \frac 1 t \int_0^t T_s f \d s -f)=0 $ pointwise on $E$.
We have $
V (S_t \frac 1 t \int_0^t T_s f \d s -f)= \int_t^\infty S_{s'} (\frac 1 t \int_0^t (T_s f-f) \d s )\d s' - \int_0^t  S_{s'} f \d s'. 
$
Since $f\in bC(E)$, the second term from the right hand side of the last equality tends to zero when ${t\searrow 0}$.
For the first term we have the estimation 
$ \left| \int_t^\infty S_{s'} (\frac 1 t \int_0^t (T_s f-f) \d s )\d s' \right| \leqslant 
V(\frac 1 t \int_0^t |T_s f-f| ds)
$
and because  $\lim_{s \searrow 0} T_s f = f$ pointwise on $E$, the first term also vanishes when ${t\searrow 0}$.
\end{proof}

%\begin{exam} \label{exam4.3}  % Example 4.3
%{\rm
%{\bf The subordination of a right continuous flow.}

\subsection{Right continuous flow driving its subordinate process} \label{subesec4.1}

Let $\Phi=(\Phi_t)_{t\geqslant 0}$  be a  right continuous flow on $E$, with transition function  $\Ss=(S_t)_{t\geqslant 0}$. 
% and such that  the mapping $(x, t)\longmapsto \Phi_t (x)$ is continuous on $E$. 
 Let further $\mu=(\mu_t)_{t\geqslant 0}$ be a convolution semigroup on $\mathbb{R}_+$
 and consider $\Ss^\mu = (S_t^\mu)_{t\geqslant 0}$, 
the {subordinate of $(S_t)_{t\geqslant 0}$  in the sense of Bochner } w.r.t. $\mu$, %as defined in Example \ref{exam4.2}.
%is the transition function $\Ss^\mu = (S_t^\mu)_{t\geqslant 0}$ on $E$ 
defined as
$S^\mu_t f := \int_0^\infty \!\!\!  S_s f \! \! \ \mu_t(\! \d s),$   $t\geqslant 0,$ $f \in p\cb(E)$;
for details see e.g. \cite{SchSoVo12} and also \cite{Lu14}.
%The probabilistic counterpart is a  procedure of introducing jumps in the evolution of the given right Markov process $X$, by means of the positive real-valued  stationary stochastic process $({\xi_t})_{t\geqslant 0}$, 
%with path space $\Omega'$, with independent nonnegative increments (called {\it subordinator}), induced by  $\mu=(\mu_t)_{t\geq 0}$.
In particular,  the subordinate process $Y^\xi=(Y^\xi_t)_{t\geqslant 0}$  is defined as 
$$
Y^\xi_t(x, \omega) := \Phi_{\xi_t(\omega)}(x), \,\,\,  {t\geqslant 0}, (x, \omega) \in E\times \Omega
$$
and it turns out that $Y^\xi=(Y^\xi_t)_{t\geqslant 0}$ 
is a right Markov process with state space $E$, path space  $E\times \Omega'$, and transition function $\Ss^\mu =(S_t^\mu)_{t\geqslant 0}$,
where $\Omega'$ is the path space of the {\it subordinator}   $({\xi_t})_{t\geqslant 0}$,  
the positive real-valued  stationary stochastic process 
with path space $\Omega'$, with independent nonnegative increments  induced by  $\mu=(\mu_t)_{t\geqslant 0}$.
So, $Y^\xi$ is obtained by
introducing jumps in the evolution of the given right continuous flow $\Phi$, by means of the subordinator 
induced by  $\mu=(\mu_t)_{t\geqslant 0}$.
%}
%\end{exam}

We state now a consequence of Theorem \ref{thm4.1} involving the right continuous flow $\Phi$ and the subordinate process $Y^\xi$. %, using also Lemma \ref{lem4.1}.

\begin{cor} \label{cor4.4} % Corollary 4.4

Let $\Ss=(S_t)_{t\geqslant 0}$ be the transition function of   a right continuous flow $\Phi=(\Phi_t)_{t\geqslant 0}$ on $E$.
Let $({\xi_t})_{t\geqslant 0}$ be  a positive real-valued  stationary stochastic process 
with independent nonnegative increments induced by  a  convolution semigroup $\mu=(\mu_t)_{t\geq 0}$ on $\mathbb{R}_+$.
Further, define
$$
{Y}_t^\Phi : = \Phi_{t+\xi_t}, \, t\geqslant 0.
$$ 
Then the following assertions hold.
\begin{enumerate}
    \item[(i)]  ${Y}^\Phi :=(E \times \Omega, {Y}^\Phi_t )$ is a right Markov process with state space $E$ and the transition function  ${\Tt}^\Phi:= ({T}^\Phi_t)_{t\geqslant 0}$ defined as  ${T}_t^\Phi :=S_t S^\mu_t$ for all $t\geqslant 0$.
    \item[(ii)] Let $(D, \cd(D))$,  $(D^\mu , \cd(D^\mu ))$,  and $({L}^\Phi, \cd({L}^\Phi))$  be the extended weak generators of  $\Ss$,   $\Ss^\mu$, and respectively ${\Tt}^\Phi$.  
    Let further   $\cd_o:= V^\mu_\alpha V_\beta(bC(E)),$    $\alpha, \beta >0$, where 
    $\cv=(V_\alpha)_{\alpha >0}$    (resp.  $\cv^\mu =(V_\alpha^\mu)_{\alpha >0}$)  is the resolvent of $\Ss$ (resp. the resolvent of $\Ss^\mu$).
    Then %$\cd_o\subset \cd_o(L)\cap \cd_o (D) \cap  \cd_o({L}^\Phi)$,
    $\cd_o \subset \cd_o^c(D^\mu)\cap \cd_o(D) \cap \cd (L^\Phi)$  and  
    $$
    {L}^\Phi= D^\mu + D \mbox{ on } \cd_o.
    $$
\end{enumerate}

\end{cor}

\begin{proof}
We apply Theorem \ref{thm4.1} for  $X:= {Y}^\Phi $.
We clearly have $X^\Phi_t = \Phi_t(Y^{\xi}_t)= \Phi_t (\Phi_{\xi_t})=\Phi_{t+\xi_t}$ and 
observe  that the paths $t\longmapsto \Phi_{t+\xi_t(\omega)}(x)  $ are right continuous, without assuming that the right continuous flow $\Phi$ is continuous.
For all $t, t' >0$ we have
$S_{t'}  S^\mu_{t} = S^\mu_{t}S_{t'} = \int_0^\infty S_{s+t'} \mu_t(\! \d s).$
%hence we cane take into account Lemma \ref{lem4.1}. 

Assertion $(ii)$ follows from Theorem \ref{thm4.1} $(ii)$.
\end{proof}

%\begin{exam}  \label{exam4.5} % Example 4.5
%{\rm {\bf Continuous flow driving a superprocess.}

\subsection{Continuous flow driving a superprocess} \label{subesec4.2}

Let $\Psi:[0,\infty)\rightarrow {\R}$ be a  branching mechanism,
$$
\Psi(\lambda)=-b\lambda-c\lambda^{2}+\int^{\infty}_{0}(1-e^{-\lambda
s}-\lambda s) N(\! \d s), 
$$
where $b, c\in \R,$  $c \geqslant 0$,  
and  $N$ is a measure on $(0, \infty)$  such that $N(u\land u^2)< \infty$.
Consider  the superpocess $\widehat{X^0}$ on   
the set $M(E)$ of all positive finite measures on $E$, 
having the branching mechanism
$\Psi$ and having no spatial motion.  
According to \cite{BeVr21} the superprocess $\widehat{X^0}$ is called {\it pure branching}.
For details on the measure-valued branching processes see   \cite{Da93}, \cite{Dy02}, \cite {LeG99} and also \cite{Fi88}, \cite{BaBe16}, and \cite{Be11}.

Let further $\Phi$ be a continuous flow  on $E$ and
consider the superprocess   $\widehat X$ on $M(E)$,  
having the spatial motion $\Phi$ and the branching mechanism $\Psi$. 
By $\Phi$ we also denote the continuous flow on $M(E)$ (endowed with the weak topology) canonically induced by the given flow $\Phi$ on $E$.

It turns out  that one can apply Theorem \ref{thm4.1} on $M(E)$ for $\widehat{X^0}$ instead of $X$ and the flow $\Phi$ on $M(E)$.
We get the following representation of the superprocess $\widehat X$ by means of the pure branching superprocess 
$\widehat{X^0}$:
$$
\widehat{X}_t  = \Phi_t (\widehat{X^0_t})  \mbox{ for all } t\geqslant 0,
$$
where the equality is in the distribution sense; see \cite{BeVr21}.
A similar result holds for non-local branching processes (in the sense of  \cite{BeLu16} and \cite{BeLu-StVr20}) on the set of all finite
configurations of the state space of the spatial motion; 
see also \cite{BeLu-StTeod24} for an associated nonlinear Dirichlet problem.
%}
%\end{exam}

\section{Multiplicative $L^p$-semigroups and continuous flows} \label{sec5} % Section 5

Let $\Phi=(\Phi_t)_{t\geqslant 0}$  be  a  semi-dynamical system  with state space $(E,\mathcal{B})$, 
$\Ss=(S_t)_{t\geqslant 0}$ its transition function,
and   $\mathcal{U}=(U_\alpha)_{\alpha > 0} $ be   the associated resolvent of kernels.
Let further $m$ be a positive $\sigma$-finite measure on $E$ which 
{\it subinvariant} for $\Ss$, that is,
\begin{equation*}
    m\circ S_t\leqslant  m \mbox{ for all } t>0,
\end{equation*}
and fix $p\in [1, \infty)$.
Then each kernel $S_t$, $t\geqslant 0$, induces a contraction on $L^p(E, m)$ which is {\it Markovian}, 
that is, if  $f\in L^p(E, m)$,  $0\leqslant f\leqslant 1$ then  $0 \leqslant S_t f \leqslant 1$ and there exists a sequence $(f_n)_n\subset L^p(E, m)$, $f_n\leqslant 1$ for all $n$, such that the sequence $(S_t f_n)_n$ is increasing $m$-a.e. to the constant function $1$.
It turns out that  {\it 
 \begin{equation}\label{5.1}
 \mbox{\it  the transition function  $\Ss$ %=(S_t)_{t\geqslant 0}$ 
 of a semi-dynamical system becomes a $C_0$-semigroup}
 \end{equation}
 
 \vspace{-2mm}
 
 \noindent
of Markovian contractions on $L^p(E, m)$ which  in addition is {\rm multiplicative on $L^p(E, m)$}, i.e.,
$$
S_t (fg) = (S_t f)(S_t g) \  \mbox{ for all } \ f, g \in L^\infty (E, m)\cap L^p(E, m)\  \mbox{ and }\  t\geqslant 0.
$$
}

In this framework, Theorem \ref{thm4.1} has a natural correspondent which goes as follows:
 
\begin{prop} \label{prop5.1} % Propoaition 5.1
Let $\Tt=(T_t)_{t\geqslant 0}$ be the transition function of a right Markov process 
$X=(\Omega, \cf, \cf_t, X_t, \Pp^x)$ 
with state space $E$ and  
$\Phi=(\Phi_t)_{t\geqslant 0}$  a  right continuous flow on $E$, with transition function 
$\Ss=(S_t)_{t\geqslant 0}$ as in Theorem \ref{thm4.1}. 
Let $m$ be a positive $\sigma$-finite measure on $E$ which is
{\it subinvariant} for both $\Ss$ and  $\Tt$ and let $p\in (1, \infty)$. 

Consider  the generators $(L_p, \cd(L_p))$,  $(D_p, \cd(D_p))$, and $({L}_p^\Phi, \cd({L}_p^\Phi))$
of  $\Tt$, $\Ss$,  and respectively   ${\Tt}^\Phi$as $C_0$-semigroups on $L^p(E, m)$,
where $\Tt^\Phi=(T^\Phi_t)_{t\geqslant 0}$ is defined as 
\begin{equation*}
    T^\Phi_t:= S_t T_t \mbox{ for all } t\geqslant 0.
\end{equation*}
Let further   $\cd_o:= U_\alpha V_\beta(L^p(E, m)),$    $\alpha, \beta >0$, where $\cu$  and $\cv$ are the resolvents
 of $\Tt$ and $\Ss$ on $L^p(E, m)$.
Then  the following assertions hold.
\begin{enumerate}
    \item[(i)] $\cd_o$ is a core of $L_p$ and $D_p$, $\cd_o \subset \cd(L_p)\cap \cd(D_p) \cap \cd (L_p^\Phi)$, and 
    $$ 
    {L}_p^\Phi= L_p + D_p  \mbox{ on } \cd_o.
    $$
    \item[(ii)] Let $X^\phi_t:= \Phi_t(X_t)$, $t\geqslant 0$, $g_0\in L_+^{p'}(E,\mu)$ (where $\frac 1p + \frac{1}{p'}=1$) be such that $\int_E g_0 \d m =1$, and put $\nu=g_0\! \cdot \! m$. 
    Then $(X_t^\Phi)_{t\geqslant 0}$ solves the martingale problem for $\bigl(L^\Phi_p,\cd(L^\Phi_p)\bigr)$ under  $\Pp^\nu=\int_E \; \Pp^x\nu(\d x)$, that is, for every $u\in \cd(L^\Phi_p)$
    $$
    \left( u(X^\Phi_t)- u(X^\Phi_0)- \int_0^t L^\Phi_p u(X^\Phi_s)\d s \right)_{t\geqslant 0}
    $$
    is an $(\Fscr_t)_{t\geqslant0}$-martingale under $\Pp^\nu$.
\end{enumerate}
\end{prop}

\begin{proof}
Because $\cd(L_p)=U_\alpha( L^p(E, m) )$  and $\cd(D_p)= V_\beta ( L^p(E, m) )$ we clearly have that
$D_o$ is dense in $L^p(E, m)$. 
Assertion  $(i)$ follows arguing as in the proof of  Theorem \ref{thm4.1} $(ii)$.

Assertion $(ii)$ is a consequence of Proposition 1.4 from \cite{BeBoRo06a}.
\end{proof}

Proposition \ref{prop3.5} has an $L^p$-version as well, and its proof is given in Appendix {(A.5)}.

\begin{prop} \label{prop5.2} % Proposition 5.2
Let $(P_t)_{t\geqslant 0}$ be a sub-Markovian strongly continuous semigroups of contractions on $L^p(E,\mu)$. 
Then the following assertions are equivalent.
\begin{enumerate}
    \item[(i)] The semigroup $(P_t)_{t\geqslant 0}$  is multiplicative on $L^p(E, m)$.
    \item[(ii)] If $(L,D(L))$ is the infinitesimal generator of $(P_t)_{t\geqslant 0}$, then 
    $$
    u\in D(L)\cap L^{\infty}(E,\mu)\Rightarrow u^2\in D(L) \quad \mbox{and}\quad  Lu^2=2uLu.
    $$
\end{enumerate} 

\end{prop}

\paragraph{Example.} 
Let $E=[0,1)\cup(1,\infty)$, $\mu=$ Lebesgue measure on $E$ and for $f\in L^{p}(E,\mu)$, let $P_tf:=f(.+t)$. 
Then $(P_t)_{t\geqslant 0}$ is a sub-Markovian $C_0$- semigroup of contractions on $L^{P}(E,\mu)$ which is multiplicative. 
Let $E'=[0,\infty)$. Then clearly $(P_t)_{(t\geqslant 0)}$ coincides (on $L^p$) with the transition function of the semi-dynamical system on $E'\supset E $ given by uniform motion to the right.

\medskip
The next theorem is the main result on multiplicative $L^p$-semigroups and continuous flows, and it represents a converse of statement \eqref{5.1}.

\begin{thm}  \label{thm5.3} %Theorem 5.3
Let $p\in [1,+\infty)$ and $({\bf S}_t)_{t\geqslant 0}$ be a 
$C_0$-semigroup of Markovian  contractions on $L^p(E, \mu )$ which is multiplicative, 
where $(E,\cb)$ is a Lusin measurable space and $\mu$ is a
$\sigma$-finite measure on $(E,\cb)$. 
Then there exist a Lusin topological space $E'$ with $E\subset E'$, $E\in \cb'$ 
(the $\sigma$-algebra of all Borel subsets of $E'$), $\cb=\cb'|_{E}$,
and a continuous flow with state space  $E'$ such that its
transition function $\Ss=(S_t)_{t\geqslant 0}$,
regarded on $L^p(E',\overline{\mu}),$ coincides with $({\bf S}_t)_{t\geqslant 0}$,
where $\overline{\mu}$ is the measure on $(E',
\cb')$ extending $\mu$ by zero on $E'\setminus E$.
\end{thm}

\begin{proof}
Let $({\bf V}_\alpha)_{\alpha >0}$ be the resolvent of sub-Markovian contractions on $L^p(E, \mu)$ associated with $({\bf S}_t)_{t\geqslant 0}$.
By  Theorem 2.2 from \cite{BeBoRo06} there exist a Lusin topological space $E' $ with $E\subset E'$, $E\in \cb'$ 
(the $\sigma$-algebra of all Borel subsets of $E'$), $\cb=\cb'|_{E}$,
and a right Markov process $X$ with state space  $E'$ such that its
resolvent $({V}_\al)_{\al>0}$, regarded on $L^p(E',{\mu}'),$ coincides with
$({\bf V}_\al)_{\al>0}$, where ${\mu}'$ is the measure on $(E',
\cb')$ extending $\mu$ by zero on $E'\setminus E$.
% and $E$ is a finely dense subset of $E'$. 

Let $(P'_t)_{t\geqslant 0}$ be the transition function of $X$ and $\ca$ be a countable subset of   $bp\cb'\cap L^p(E', \mu')$  which is multiplicative and generates the $\sigma$-algebra $\cb'$. 
Consider the set
$$
F_o = \{ x\in E' :   P'_t( V_\beta f \cdot V_\beta g) = P'_t( V_\beta f) P'_t (V_\beta g) \mbox{ for all } t\in \Q_+ \mbox{ and } f, g\in \ca \} 
$$
for some $\beta>0$. 
Clearly, $P'_t$ coincides with ${\bf S}_t$ as an operator on $L^p(E', \mu')$ for each $t\geqslant 0$,  
hence it is multiplicative on $L^p(E', \mu')$ and therefore $\mu' (E'\setminus F_o)=0$.  
We have $F_o\in \cb'$ and applying Lemma 2.8 from \cite{BeCi16} we deduce that it is finely closed.
By Lemma 2.1 and its proof from \cite{BeBoRo06} there exists a finely closed set $F\in \cb'$, $F\subset  F_o$, such that 
$\mu' (E'\setminus F)=0$ and $V_\alpha (1_{E' \setminus F})=0$  on $F$. 
Since  $V_\alpha (1_{E'\setminus F})>0$ on the finely open set $E'\setminus F$,  
if follows that $F$ is an absorbing subset of $E'$.
Therefore we may consider the restriction $(P_t)_{t\geqslant 0}$ of the transition function  $(P'_t)_{t\geqslant 0}$ from $E'$ to $F$,
$P_t f: = P'_t {f}' |_F$, where $f'\in p\cb'$ is such that $f' |_F= f$.

Because the functions $t \longmapsto    P'_t( V_\beta f \cdot V_\beta g)$ and 
$t \longmapsto   P'_t( V_\beta f) $ are right continuous on $[0, \infty)$ it follows that 
$P'_t( V_\beta f \cdot V_\beta g) = P'_t( V_\beta f) P'_t (V_\beta g)$  on $F$ for all $t\geqslant 0$  and  $f, g\in \ca$.
By a monotone class argument we get that $(P_t)_{t\geqslant 0}$ is a multiplicative transition function on $F$ and condition $(\ref{eq2.1})$ is satisfied.
Consequently, Remark \ref{rem2.1} implies that there exists a semi-dynamical system $\Phi^o =(\Phi^o_t)_{t\geqslant 0}$ on $F$  having the transition function  $(P_t)_{t\geqslant 0}$.

Let  $\Phi=(\Phi_t)_{t\geqslant 0}$ on $E'$ be the trivial extension of $\Phi^o$ from $F$ to $E'$,
$\Phi_t (x):= \Phi^o_t(x)$ if $x\in F$ and $\Phi_t (x)= x$ if $x\in E'\setminus E$ for all $t\geqslant 0$.
Since $(sd4)$ holds on $F$ for $\Phi^o$ with the countable set $\cC_o\subset bp\cb$ then 
$(sd4)$ also holds  for $\Phi$ on $E'$ considering a countable set $\cC'_o\subset bp\cb'$ which separates the points of $E'$
and $\cC'_o|_E = \cC_o$. 
So, $\Phi=(\Phi_t)_{t\geqslant 0}$  is a semi-dynamical system on $E'$ and applying Theorem \ref{thm2.2} 
we may replace the topology of $E'$ with a a conveninent Ray one, such that $\Phi$ becomes a continuous flow on $E'$ as claimed.
\end{proof}

\begin{rem} % Remark 5.4
It is proven in \cite{BeBoRo06a} that under additional assumptions on the domain of the generator of a $C_0$-semigroup of sub-Markovian contractions 
on $L^p(E, m)$ the associated Markov process exists on $E$, so, it is not more necessary to consider a larger state space; for applications in significant examples see also
\cite{GrNo20}, \cite{GrWa19} and \cite{CoGr10}.
In this case, if the semigroup is multiplicative on $L^p(E, m)$, one can see  that the associated continuous flow from Theorem \ref{thm5.3} remains on $E$.
\end{rem}

\section*{Appendix}
\begin{proof}[(A.1) Proof of \Cref{lem3.1}]
   Observe first that if $f\in \cb_o$ then $U_\alpha|f|\leqslant U_\alpha h_\alpha< \infty $ for all $\alpha >0$, so, $U_\alpha f\in [\cb]$.

\medskip
$(i)$  Let $\alpha, \alpha' >0 $ and $\alpha_o: = \inf (\alpha, \alpha ')$. 
If $f\in\cb_o$ then there exist $t_o >0$ and $h_{\alpha_o} \in p\cb$ such that
 $\sup_{0<s< t_o} T_s | f | \leqslant h_{\alpha_o}$  with  $U_{\alpha_o}  h_{\alpha_o} < \infty$.
 Then one can see that
 $\sup_{0<s< t_o} T_s |U_\alpha f |  \leqslant U_\alpha h_{\alpha_o}$. 
 Since $\alpha_o\leqslant \alpha , \alpha'$ and $U_{\alpha_o}  h_{\alpha_o} < \infty$, it follows that
 $U_{\alpha}  h_{\alpha_o}$ and $U_{\alpha'} U_\alpha h_{\alpha_o}$ are also real-valued functions. 
 We conclude that $U_\alpha f\in \cb_o$. 
 We have also $T_t U_\alpha h_{\alpha_o}\leqslant e^{\alpha t}U_\alpha h_{\alpha_o}<\infty$, $t>0$,
 hence $U_\alpha f\in \cb_{oo}$.
 
 Let  now $f\in\cb_{oo} $,   $\alpha, t'>0$ and $t_o > 0$  be such that 
$\sup_{0<s< t_o} T_s | f | \leqslant h:=\inf( h_t, h_{t+t'},  h_\alpha)  \in p\cb$  with   
$T_t h_t  +T_{t+t'} h_{t+t'}  + U_\alpha h_\alpha<\infty$. 
Then
 $ T_t |T_s  f |  \leqslant T_t h_t <\infty$ for every $s<t_o$.
 Since $\lim_{s\searrow 0}  T_s f=f$, we deduce by dominated convergence that
 $T_t|f|<\infty$ and $ \lim_{s\searrow 0}  T_s T_t f= T_t f$.
 We have also $\sup_{0<s< t_o}  T_s|T_t  f |  \leqslant T_t h <\infty$ with $U_\alpha T_t h \leqslant e^{\alpha t} U_\alpha h <\infty$ and $T_{t'} T_t h<\infty.$
 Therefore   $T_t f\in \cb_{oo}$.

\medskip
Assertion $(ii)$ follows  because $U_\alpha h, T_t h\in \ce$, provided that $h\in \ce$.

\medskip
$(iii)$ Let  $f\in\cb^0$ be  bounded, so, we may assume that $|f|\leqslant 1$. 
Then $| f | \leqslant \widehat{1}:= \lim_{t\searrow 0} T_t 1$ which is excessive and  therefore $f$ belongs to $\cb_e$.

\medskip
$(iv)$ The first assertion follows from $(ii)$ since  $ f = \lim_{t\searrow 0} T_t f$ if $f\in[\cE]\cup b[\ce_\alpha]$.
If $U|f|<\infty$ then $Uf\in [\cE]$ and therefore $Uf \in \cb_{oo}$.
\hfill $\square$  
\end{proof}

\begin{proof}[(A.2) Proof of \Cref{prop3.3}]

\medskip
$(i)$ 
Since by assertion $(i)$ of Lemma  \ref{lem3.1} we have $U_\alpha (\cb_o)\subset \cb_{oo}$, 
it is clear that the set $U_\alpha (\cb_o)$ does not depend on $\alpha>0$.
Let $u=U_\alpha f$ with $f\in \cb_o$. 
Then $u$ also belongs to $\cb_o$, hence in particular,  $U_\alpha |u|<\infty$.
Let further $t_o>0$ and $h_\alpha \in p\cb$ be such that 
$T_s |f|  \leqslant h_\alpha$ for all $s<t_o$ and $U_\alpha h_\alpha <\infty$.
%We may suppose that $\beta> \alpha$.
We have 
$| {T_t u-u} |\leqslant  (e^{\alpha t} -1)  |u| +  h_\alpha e^{\alpha t} \int_0^t e^{ -\alpha s} \d s$ if $t<t_o$.
Therefore
$ \sup_{0<t<t_o} |\frac{T_tu- u}{t }| \leqslant h:= \alpha |u| + h_\alpha$ and $U_\alpha h<\infty$.
We also have 
$\frac{T_t u-u}{t} = \frac{e^{\alpha t}- 1}{t} u + \frac{ 1- e^{\alpha t} }{\alpha t} f - \frac{e^{\alpha t}} {t} \int_0^t e^{-\alpha s} (T_s f - f) \d s.$
Clearly, when $t\searrow 0$, the first  term from the right hand side converges pointwise to $\alpha u$, 
the second one  to $-f$, while the third one converges to zero because
$\lim_{s \searrow 0} T_s f=f$. We conclude that $u\in \cd(L)$ and $Lu= \alpha u -f.$
Conversely, if $u\in \cd(L)$ then let $\alpha, t_o>0$,  and $h_\alpha \in p\cb$ with $U_\alpha h_\alpha <\infty$ and
$\sup_{0<t<t_o} | \frac{T_tu- u}{t } |\leqslant h_\alpha$.
Let  $v:= Lu= \lim_{t \searrow 0}\frac{T_t u-u}{t}\in \cb_o$.
Because $U_\alpha h_\alpha<\infty$, by dominated convergence we get
$\lim_{t \searrow 0}\frac{T_t U_\alpha u-U_\alpha u}{t}= U_\alpha v$. 
On the other hand, from the first part of the proof  we have $U_\alpha u\in \cd(L)$ and 
$\lim_{t \searrow 0}\frac{T_t U_\alpha u-U_\alpha u}{t}= L(U_\alpha u)= \alpha U_\alpha u- u$. 
We conclude  that $u=U_\alpha (\alpha u- v)\in U_\alpha(\cb_o)$.

To prove   the last assertion of  $(i)$  we argue as in the proof of Proposition 1.5 (a) from \cite{EthKu86}. 
We have $\frac{T_h u -u}{h} = \frac 1 h \int_0^t [ T_{s+h} f - T_s f] \d s=
\frac 1 h \int_t^{t+h} T_s f \d s - \frac 1 h \int_0 ^h Ts f \d s$.
Because the function $s\longmapsto T_s f(x)$ is right continuous on $[0, \infty)$ for every $x\in E$, 
it follows  that
$\lim_{h\searrow 0} \frac{T_h u -u}{h}= T_t f -f$ pointwise on $E$.
Since we also have  $| \frac{T_h u -u}{h} |\leqslant 2 \| f\|_\infty$ for all $h>0$, we conclude that $u$ belongs to $\cd(L)$ and $Lu=T_t f -f$.

\medskip
$(ii)$ 
Let $x\in E$. Since $g\in \cb^0$ we have  $\lim_{t\searrow 0}  T_tg(x)= g(x)$ and therefore 
$\lim_{t\searrow 0} \frac{T_tu (x)- u(x)}{t} = \lim_{t\searrow 0}  \frac{1}{t} \!\!\int_0^t  T_s g(x)\! \d s$ $=g(x)=\overline{L} u (x).$ 

\medskip
$(iii)$ We clearly have $\cd(L_w)\subset \cd_e(L)$ because  $b\cb^0 \subset \cb_e$.
Let $u=U_\alpha f\in \cd(L)$.
Then, by assertion $(i)$  of Lemma \ref{lem3.1} we get $u \in \cb_{oo}$ and we have
$\int_0^t T_s(\alpha u -f)= \alpha \int_0^t  \int_0^\infty e^{-\alpha r} T_{r+s} f \d r\d s - \int_0^t T_s f \d s=
\int_0^\infty (e^{\alpha r\wedge t}-1)  e^{-\alpha r} T_t f  \d r-  \int_0^t T_s f \d s= -u +e^t \int_t^\infty  e^{-\alpha r} T_r f \d r= -u+  T_t U_\alpha f =  -u + T_t u$, where for the second equality we used Fubini's Theorem. 
We conclude that $u\in \cd(\overline L)$ and  by assertion $(ii)$ we  clearly  have
$\overline L u= Lu$.
Let now $u\in \cd(\overline{L})\cap \cb_o$ such that $\overline{L} u\in \cb_o$, let $\alpha >0$ and $h_\alpha\in p\cb$ with $U_\alpha h_\alpha < \infty$,
be such that $T_s (|Lu|)\leqslant h_\alpha$ for all $s <  t_o$ for some $t_o >0$.
Then $| \frac{T_tu - u }{t} |\leqslant \frac 1 t \int _0^t T_s(|Lu|) \! \d s 
\leqslant  h_\alpha$ for all $t<t_o$.
It follows   that $u\in \cd (L)$.

\medskip
$(iv)$ Let $u\in \cd_o(L)$ and $\alpha>0$.
Then $u=U_\alpha (\alpha u- Lu)$ with $u, Lu \in \cb_{oo}$, so, $u\in U_\alpha (\cb_{oo})$.
Conversely, if $u=U_\alpha f$ with $f\in \cb_{oo}$, then by assertion $(i)$ we have $u\in \cd(L)$
and $Lu= \alpha u-f\in \cb_{oo}$, hence $u\in \cd_o(L)$.

Let $u=U_\alpha f \in \cd_o(L)$, $ f\in \cb_{oo}.$ 
According to  Lemma \ref{lem3.1} $(i)$ we get $T_t f\in \cb_{oo}$.
Therefore
 $T_t u= U_\alpha T_t f$ also belongs to $\cd_o(L)$  
and we have 
$LT_t u= LU_\alpha T_t f= \alpha U_\alpha T_t f - T_t f= T_t Lu$.

\medskip
The proof of  $(v)$ is straightforward. 

\medskip
Assertion $(vi)$ follows arguing as in the proof of $(iv)$ and using Lemma \ref{lem3.1} $(ii)$.

\medskip
$(vii)$ The first inclusion follows from assertion $(iv)$.
Let now  $u \in U_\alpha(\cd(L)$, $u=U_\alpha U_\beta f$ with $f\in \cb_o$.
Then $LT_t u= \alpha T_t u - T_t U_\beta f$ and it is continuous in $t$,  
according with the following remark:
{\it If $g\in[\cb]$ is such that $U_\alpha |g|<\infty$ then the real-valued function $t\longmapsto T_t U_\alpha g (x)$ is continuous on 
$[0, \infty)$ for each $x\in E$ because $T_t U_\alpha g = e^{\alpha t} \int_t^\infty e^{-\alpha s} T_s g \d s.$ }

To prove the last inclusion of assertion $(vii)$, observe that by Lemma \ref{lem3.1} $(iv)$ 
we have $U_\alpha(b[\cb])\subset \cb_{oo}$ and by assertion $(iv)$
 we obtain $U_\beta U_\alpha ({b[\cb]})\subset  \cd_o(L)$.
The continuity property is obtained using again the above remark.

\medskip
$(viii)$  Let $u\in \cd^c_o(L)$, $u=U_\alpha f$ with $f\in \cb_{oo}$.
Then  by Lemma \ref{lem3.1} $(i)$ we have  $T_t f\in \cb_{oo}$ for each $t\geqslant 0$ and $L T_t u= \alpha T_t u - T_t f$.
Because $t\longmapsto T_t u(x)$ is continuous, it follows that $T_t f(x)$ is also continuous in $t$ on $[0, \infty)$ for each $x\in E$.
We have $T_t u= e^{\alpha t} (u-\int_0^t e^{\alpha s} T_s f \d s)$ and from the above considerations the first statement of assertion $(viii)$ follows.
In particular, we proved that $u_t:=T_t u$, $t\geqslant 0$, is a solution to the equation $(\ref{eq3.3})$, satisfying the requested conditions: 
$T_0=u$, $\|T_t u\|_\infty\leqslant \| u\|_\infty,$
$T_t u \in \cd_o(L)$ by the above assertion $(iv)$, $LT_t  \in \cb_{oo}$, 
and  $LT_t u(x)$  is continuous in $t$ because we assumed that $u$ belongs to $\cd^c_o(L)$.

We show  now the uniqueness property for the solution to the equation $(\ref{eq3.3})$ and as announced, 
we use a classical argument, e.g., as in the proof of Theorem 1.3 from \cite{Dy65}, Ch. I, section 3, page 28.
Let $u_t$, $t\geqslant 0$, be a solution of  $(\ref{eq3.3})$ such that $u_0=0$, 
$u_t \in \cd_o(L)$, $\| u_t \|_\infty$ is bounded, $Lu_t\in \cb_{oo}$, and $Lu_t(x)$  is continuous in $t$ for each $x\in E$.
We have to show that $u_t=0$ for each $t>0$.
Let $\alpha >0$ and $v_t:= e^{-\alpha t}u_t$. 
Then  $\frac{\d v_t}{\hspace{-1.5mm}\d t}= (L-\alpha)v_t$ with $v_t\in \cd_o(L)$. 
It follows that $U_\alpha (\frac{\d v_t}{\hspace{-1.5mm}\d t})= -v_t$  for each $t>0$ and therefore
$\int_0^t v_s \d s= - U_\alpha (\int_0^t \frac{\d v_s}{\hspace{-1.5mm}\d s} \d s )= -U_\alpha v_t$.
Consequently, $\int_0^t e^{-\alpha s} u_s (x) \d s=  - e^{-\alpha t} U_\alpha u_t(x)$. 
Since $\| u_t \|_\infty$ is bounded, letting $t\to \infty$, it follows that the right hand side of the above equality tends to zero.
We conclude that $\int_0^\infty e^{-\alpha s} u_s (x) \d s=  0$ for every $\alpha >0$ and $x\in E$ and therefore $u_s(x)=0$ for each $s>0$ and $x\in E$.
\end{proof}

\begin{proof}[(A.3) Proof of Cref{prop3.6}]

Let $g=\overline D f$ with $f\in \cd(\overline D)$  and $\int_0^t  S_s (| f\overline D f|)  \d s < \infty$ for all $t>0$.
We have to prove that $S_tf^2= f^2 + 2\int_0^t S_s f S_s g \d s$ for all $t>0$, provided that
$S_t f= f+ \int_0^t S_s g \d s$.
Indeed, we have 
$\int_0^t S_s f S_s g \d s =2\int_0^t [ f+ \int_0^s S_u g \d u] S_s g \d s=
f\int_0^t S_s g \d s + \int_0^t \d s S_s g \int_0^s S_u g \d u=
f\int_0^t S_s g \d s + \int_0^t \d u S_u g [\int_0^t S_s g \d s - \int_0^u S_s g\d s] =
f\int_0^t S_s g \d s + \int_0^t \d u S_u g [\int_0^t S_s g \d s + f -S_u f ] =
2f\int_0^t S_sg \d s - \int_0^t S_u f S_u g \d u = S_t f^2 + f^2 - 2 f S_t f$.
We conclude that 
$ 2 \int_0^t S_s f S_s g \d s =2f (S_t f - f) + S_t f^2 + f^2 -2f S_t f= S_t f^2 - f^2.$    
\end{proof}

\begin{proof}[(A.4) Proof of \Cref{prop3.8}]

\medskip
The proof of $(i)$ is a straightforward verification.

\medskip
$(ii)$ Let $ u \in \cd({D^o})$ and $x\in \co^c$. 
Then by $(i)$ we have $\Phi^o_t (x)=x$ and therefore ${D^o} u(x)=0$.
Let further $\Ss=(S_t)_{t\geqslant 0} $ (resp. $\Ss^o= (S^o_t)_{t\geqslant 0} $) be the transition function of $\Phi$  (resp.  of $\Phi^o$). 
If $u\in p\cb(E)$ the $S^0_t u(x) = S_t u(x)$ provided that $t<T(x)$ and
$S^0_t u(x)= u(\Phi_{T(x)} (x)$ if $t>T(x)$  and $T(x)<\infty$.
If  $u\in \cd(D)$ and 
$x\in \co$ then there exists $\varepsilon >0$ such that 
$\Phi_t (x)\in \co$ for all $t \leqslant \varepsilon$, hence $T(x)\geqslant \varepsilon$  and therefore 
$S^o_t u(x)= S_t u (x)$ for all $t\leqslant \varepsilon$. 
We conclude that 
$Du=D^o u$ on $\co$.

\medskip
$(iii)$ Let $x\in \oco$. 
If $x\in \partial \co$ then by $(i)$ we have $\Phi^o_t(x)=x\in \oco$ for all $t\geqslant 0$.
If $x\in \co$ then clearly $\Phi^o(x)=\Phi_t(x) \in \co$ for all $ t<T(x)$. 
If $t\geqslant T(x)$ then $\Phi^o_t(x)= \Phi_{T(x)} (x)\in \partial \co$ by property  $(2)$ of $T$.

\medskip
Assertion $(iv)$ follows from  $(iii)$.    
\end{proof}

\begin{proof}[(A.5) Proof of \Cref{prop5.2}]
Let $u\in  D(L)\cap L^{\infty}(E,\mu)$. 
We have 
$\frac{P_t
u^2-u^2}{t}=\frac{P_tu-u}{t}(P_tu+u)$ and since
$\frac{P_tu-u}{t}$ (resp.  $P_t u$) is converging in $L^p(E,\mu)$
to $Lu$ (resp. to $u$)  as $t\rightarrow 0$, we
deduce that $\frac{P_t u^2-u^2}{t}$ is converging to
$2uLu$, hence $u^2\in D(L)$ and $Lu^2=2uLu$.
Conversely, let $u\in D(L)\cap L^{\infty}(E,\mu)$ and put $u_t:=(P_t u)^2 \in  \cd(L)$.
Since $\frac {\d u_t}{\!\!\d t}=2 P_t u \cdot LP_t u %=L((P_t u)^2)
=Lu_t$ and $u_0= u^2$, we get that $u_t=P_t u^2$, hence $(P_t
u)^2=P_t u^2$. 
It follows that $P_t(uv)=P_t u \cdot P_t v$ for all $u,v\in D(L)\cap L^{\infty}(E,\mu)$ and
because  $D(L)\cap L^{\infty}(E,\mu)$ is
dense in $L^p(E,\mu)$ we conclude that the semigroup $(P_t)_{t\geqslant 0}$  is multiplicative on $L^p(E, \mu)$.    
\end{proof}

\paragraph{Acknowledgements.} 
This work was supported by grants of the Ministry of Research, Innovation and Digitization, CNCS - UEFISCDI,
project number PN-III-P4-PCE-2021-0921, within PNCDI III for the first author, and
project number PN-III-P1-1.1-PD-2019-0780, within PNCDI~III for the third author.

% The work of Iulian C\^{i}mpean was supported by a grant of the Romanian Ministry of Research and Innovation, CNCS--UEFISCDI, project number PN--III--P1--1.1--PD--2019--0780, within PNCDI~III.


\begin{thebibliography}{999}

%\bibitem[Ar 98]{Arn98}
%L. Arnold, {\it Random Dynamical Systems}, Springer, 1998.

\bibitem{BaBe16}
V. Barbu and L. Beznea,
Measure-valued branching processes associated with Neumann nonlinear semiflows,
{\it J.  Math. Anal.  Appl.}  {\bf 441} (2016), 167--182.

\bibitem{Be11} 
L. Beznea,  {Potential theoretical methods in the construction of measure-valued branching processes},
{\it J. European Math. Soc.} \textbf{13}  (2011), 685--707.

\bibitem{BeBo04}
 L. Beznea   and  N.  Boboc,
 {\it Potential Theory and Right Processes.} (Mathematics and Its Applications, vol. {\bf 572}),
Kluwer Academic Publishers/Springer 2004.

\bibitem{BeBo05} 
L. Beznea  and N. Boboc, 
{On the tightness of capacities associated with sub-Markovian resolvents},
{\it Bull. London Math. Soc.} \textbf{37} (2005), 899--907.

\bibitem{BeBoRo06}   
L. Beznea, N. Boboc, and M. R\"ockner, 
Quasi-regular Dirichlet forms and $L^p$-resolvents on measurable spaces, 
{\it Potential Analysis}  \textbf{25} (2006), 269--282.

\bibitem{BeBoRo06a}   
L. Beznea, N. Boboc, and M. R\"ockner, 
Markov processes associated with $L^p$-resolvents and applications to stochastic differential equations on Hilbert space,
{\it J. Evol. Eq.} \textbf{6} (2006), 745--772. 

\bibitem{BeCi16}
L. Beznea and I. C\^\i mpean,
Quasimartingales associated to Markov processes, 
{\it Trans. Amer. Math. Soc.}   {\bf 370} (2018),  7761--7787.

\bibitem{BeCiRo18}
L. Beznea, I. C\^impean, and M. R\"ockner,
Irreducible recurrence, ergodicity, and extremality of invariant measures for resolvents,
{\it Stoch. Proc. and  their Appl.}  {\bf 128}  (2018),  1405--1437.

\bibitem{BeCiRo20}
L. Beznea, I. C\^impean, and M. R\"ockner,
A natural extension of Markov processes and applications to singular SDEs, 
{\it Ann. Inst. H. Poincare Probab. Statist.} {\bf 56} (2020), 2480--2506.

\bibitem{BeCoRo11}
L. Beznea, A. Cornea, and M. R\"ockner, 
Potential theory of infinite dimensional L\'evy processes, 
{\it J. of Functional Analysis}  {\bf 261} (2011), 2845--2876.

\bibitem{BeIoLu-St21}
L. Beznea, I.R. Ionescu, and O. Lupa\c scu-Stamate,
Random multiple-fragmentation and flow of particles on a surface, 
{\it J. Evol. Equ. } {\bf 21}  (2021), 4773--4797, 
https://doi.org/10.1007/s00028-021-00732-z
%DOI 10.1007/s00028-021-00732-z 

\bibitem{BeLu16} 
L. Beznea and O. Lupa\c scu, 
{Measure-valued discrete branching {Markov} processes}, 
{\it Trans. Amer. Math. Soc.} 
{\bf 368} (2016), 5153--5176. 

\bibitem{BeLu-StTeod24} 
L. Beznea, O. Lupa\c scu-Stamate, and A. Teodor,
{\it Nonlinear Dirichlet problem of non-local branching processe}s (2024), Preprint,
arXiv:2410.16988

\bibitem{BeLu-StVr20} 
L. Beznea, O. Lupa\c scu-Stamate, and C.I. Vrabie,
Stochastic solutions to  evolution equations  of non-local branching processes, 
{\it Nonlinear Analysis} {\bf 200} (2020), 112021. 
https://doi.org/10.1016/j.na.2020.112021

%\bibitem[BeOp 11]{BeOp11}  L. Beznea and A.-G.  Oprina, 
%{Nonlinear PDEs and measure-valued branching type processes},
%{\it J.  Math. Anal.  Appl.} \textbf{384} (2011), 16--32.

\bibitem{BeRo11} 
L. Beznea and  M. R\"ockner,
{From resolvents to c\`adl\`ag processes through compact excessive functions and applications to singular SDE on Hilbert spaces}, 
{\it Bull. Sci. Math.}  {\bf 135} (2011), 844--870.

\bibitem{BeRo11a} 
L. Beznea and  M. R\"ockner,
Applications of compact superharmonic functions:
path regularity and tightness of capacities,
{\it Complex Anal. Oper. Theory}   {\bf 5} (2011), 731--741.

\bibitem{BeTr11}
L. Beznea and G. Trutnau, 
On the quasi-regularity of non-sectorial Dirichlet forms  by processes having the same polar sets, 
{\it J.  Math. Anal.  Appl.} \textbf{384} (2011), 33--48.

\bibitem{BeVr21}
L. Beznea and C.I. Vrabie, 
Continuous flows driving  branching processes and their nonlinear evolution equations,
{\it Adv. Nonlinear Anal.}  {\bf 11} (2022),  921--936. 

\bibitem{BezBu05}
M. Bezzarga and G. Bucur, 
A theorem of Hunt for semidynamical systems,
{\it J. Math. Stat.}  {\bf 1}  (2005), 58--65.

\bibitem{Bou81}
N. Bouleau,
Propri\'et\'es d'invariance du domaine du g\'en\'erateur
infinit\'esimal \'etendu d'un processus de Markov,
In {\it S\'eminaire de probabilit\'es (Strasbourg)} {\bf15} (1981), pp. 167--188.

\bibitem{CoGr10}
F. Conrad and 
M. Grothaus,
Construction, ergodicity and rate of convergence of $N$-particle Langevin dynamics with singular potentials,
{\it J. Evol. Equ.}  {\bf 10}  (2010), 623--662

\bibitem{Da93}
D.A. Dawson,
{Measure-valued Markov processes.}  
In: {\it \'Ecole d'\'Et\'e  de Probabilit\'es de Saint-Flour XXI--1991}
(Lecture Notes in Math. {\bf 1541}),   Springer (1993), pp. 1--260.

%\bibitem[Do 01]{Do01} J. L. Doob,
%{\it Classical Potential Theory and its Probabilistic Counterpart},
%Springer-Verlag 2001 (reprint of the 1984 edition).


\bibitem{Dy65}
E. B. Dynkin,
{\it Markov Processes}, Vol. I, 
Springer, 1965.

\bibitem{Dy02}
E. B. Dynkin,
{\it Diffusions, superdiffusions and partial differential equations},  
Amer. Math. Soc., Colloquium Publications, Vol. {\bf 50}, 2002.

\bibitem{EthKu86}
S.N. Ethier  and T.G. Kurtz,
{\it Markov Processes: Characterization and Convergence},
Wiley \& Sons 1986.

\bibitem{Fi88}
P. J. Fitzsimmons, {Construction and regularity of measure-valued Markov branching processes},
{\it Israel J. Math.} \textbf{64}  (1988), 337--361.

\bibitem{Fl10}
F. Flandoli, M. Gubinelli, E. Priola,Well-posedness of the transport equation by stochastic perturbation,
{\it Invent. math.} {\bf 180} (2010), 1--53.

\bibitem{GoNeRo22}
B. Goldys, M. Nendel, and M.  R\"ockner,
Operator semigroups in the mixed topology and the infinitesimal description of Markov processes, 
arXiv:2204.07484  (2022).

\bibitem{GrNo20}
M. Grothaus and  A. Nonnenmacher,
Overdamped limit of generalized stochastic Hamiltonian systems for singular interaction potentials,
{\it J. Evol. Equ.}  {\bf  20} (2020), 577--605.

\bibitem{GrWa19}
M. Grothaus and F.-Y. Wang,
Weak Poincar\'e inequalities for convergence rate of degenerate diffusion processes,
{\it Ann Probab.}  {\bf 47} (2019), 2930--2952.

\bibitem{HiYo12}
F. Hirsch and M. Yor, 
On temporally completely monotone functions for Markov processes,
{\it Probab. Surveys}, {\bf 9} (2012), 253--286.

\bibitem{Ku20}
F. K\"uhn, 
Schauder estimates for Poisson equations associated with non-local Feller generators,
{\it J. Theor. Probab.} (2020). https://doi.org/10.1007/s10959-020-01008-x

\bibitem{Kunita93}
H. Kunita, Stochastic flows and stochastic differential equations, {\it Ann. Probab.} {\bf 21} (1993), 581--587.

\bibitem{LeG99}
J.-F. Le Gall, 
\textit{Spatial Branching Processes, Random Snakes and Partial Differential Equations} 
(Lectures in Mathematics ETH Z\"urich),
 Birkh\"auser, 1999.

\bibitem{Li11}
Z. Li, 
{\it Measure-Valued Branching Markov Processes}, Second Edition
 (Probability Theory  and Stochastic Modelling, 103), Springer, 2022.

\bibitem{Lu14}
O. Lupa\c scu, 
Subordination in the sense of Bochner of $L^p$-semigroups and associated Markov processes,
{\it Acta Mathematica Sinica, English Series} {\bf 30} (2014), 187--196.

\bibitem{LyRo92}   
T. J. Lyons and M. R\"ockner, 
A note on tightness of capacities associated with Dirichlet forms, 
{\it Bull. London Math. Soc.} {\bf  24}  (1992) 181--184.


\bibitem{MaRo92} 
Z. M. Ma and M. R\"ockner, 
{\it Introduction to the Theory of (Non-Symmetric) Dirichlet Forms}, Springer, 1992.


\bibitem{SchSoVo12}
R.L. Schilling, R. Song, and Z. Vondra\v cek,
{\it Bernstein Functions, Theory and Applications} ($2^{\rm nd}$ edition), De Gruyter, 2012.

\bibitem{Sh88}
M. Sharpe, 
\textit{General Theory of Markov Processes}, 
Academic Press, Boston, 1988.
\end{thebibliography}
\end{document}